\documentclass[11pt,oneside,final]{amsart}
\usepackage{fullpage}
\usepackage[english]{babel}
\usepackage[T1]{fontenc}
\usepackage{mathtools}
\usepackage{enumerate}

\usepackage{etex}
\usepackage[utf8x]{inputenc}
\usepackage{amsmath,amssymb,amsfonts,amsthm}
\usepackage{mathrsfs}
\usepackage{hyperref}
\usepackage{tikz-cd}
\usepackage{ctable}

\theoremstyle{plain}
\newtheorem{thm}{Theorem}[section]

\newtheorem{lem}[thm]{Lemma}
\newtheorem{prop}[thm]{Proposition}

\theoremstyle{definition}
\newtheorem{rem}[thm]{Remark}

\newtheorem*{condition*}{Condition}

\newtheorem{example}[thm]{Example}
\newtheorem{conj}[thm]{Conjecture}
\newtheorem*{example*}{Example}
\numberwithin{equation}{section}
\numberwithin{thm}{section}

\DeclareMathOperator{\Gal}{Gal}

\DeclareMathOperator{\bfk}{\mathbf{k}}

\DeclareMathOperator{\bfX}{\mathbf{X}}

\DeclareMathOperator{\bfx}{\mathbf{x}}
\DeclareMathOperator{\bfy}{\mathbf{y}}

\DeclareMathOperator{\ZZ}{\mathbb{Z}}

\DeclareMathOperator{\disc}{disc}

\DeclareMathOperator{\Id}{Id}

\DeclareMathOperator{\GL}{GL}

\newcommand{\aff}{\mathrm{aff}}
\newcommand{\proj}{\mathrm{proj}}
\newcommand{\cov}{\mathrm{cov}}
\newcommand{\Abb}{\mathbb{A}}
\newcommand{\Pbb}{\mathbb{P}}

\newcommand{\dimP}{\dim_{\Pbb}}
\newcommand{\dimA}{\dim_{\Abb}}

\newcommand{\maps}{\rightarrow}
\newcommand{\ratlmaps}{\dashrightarrow}
\newcommand{\R}{\mathbb{R}}
\newcommand{\Q}{\mathbb{Q}}
\newcommand{\Z}{\mathbb{Z}}
\newcommand{\F}{\mathbb{F}}

\newcommand{\C}{\mathbb{C}}
\newcommand{\A}{\mathbb{A}}
\renewcommand{\P}{\mathbb{P}}
\renewcommand{\GL}{\text{GL}}

\newcommand{\calE}{\mathcal{E}}

\newcommand{\calO}{\mathcal{O}}

\newcommand{\bP}{\mathbb{P}}
\newcommand{\bx}{\mathbf{x}}

\renewcommand{\bx}{\mathbf{x}}
\newcommand{\bX}{\mathbf{X}}

\newcommand{\ep}{\varepsilon}
\newcommand{\al}{\alpha}
\newcommand{\del}{\delta}

\newcommand{\ga}{\gamma}
\newcommand{\om}{\omega}
\newcommand{\sig}{\sigma}
\newcommand{\beq}{\begin{equation}}
\newcommand{\eeq}{\end{equation}}

\newcommand{\con}{\equiv}
\newcommand{\modd}[1]{\; ( \text{mod} \; #1)}

\definecolor{pink}{rgb}{1,.2,.6}
\definecolor{orange}{rgb}{0.7,0.3,0}
\definecolor{blue}{rgb}{.2,.6,.75}
\definecolor{green}{rgb}{.4,.7,.4}
\definecolor{purple}{RGB}{127,0,255}
 \definecolor{gray}{RGB}{211,211,211}

\begin{document}
\title{Counting points in   thin sets:\\ A survey}

\author[Bonolis]{Dante Bonolis}
\address{Duke University, 120 Science Drive, Durham NC 27708}
\email{dante.bonolis@duke.edu}

\author[Pierce]{Lillian B. Pierce}
\address{Duke University, 120 Science Drive, Durham NC 27708}
\email{lillian.pierce@duke.edu}

\author[Woo]{Katharine Woo}
\address{450 Serra Mall, Building 380, Stanford CA 94305}
\email{khwoo98@stanford.edu }

\begin{abstract}

  In the 1980's Serre asked how many points of bounded height can lie in a thin set. This has motivated significant research ever since, culminating in a series of recent breakthroughs. It is a good time to take stock of the central questions that have been resolved, and also to highlight remaining open questions.  
  First, we survey recent progress on counting points of bounded height in the four types of thin sets, according to the projective/affine and type I/type II designations. Second, we turn to questions of uniformity.  Famously, in the setting of type I thin sets, the best-known upper bound for the number of points of bounded height is independent of the maximum size, say  $\|F\|$, of the coefficients of the polynomials that define the thin set; such an upper bound is called uniform. A uniform upper bound in the setting of type II thin sets is not known. For type II thin sets, we explore the dependence on $\|F\|$ via several strategies, and construct counterexamples that suggest the question of uniformity is quite subtle in the setting of type II thin sets.
\end{abstract}

\maketitle
 \vspace{-1cm}
\section{Introduction}\label{sec_introduction}
\subsection{Overview}\label{sec_overview}
Circa 1980, Serre formulated the notion of a thin set. A  subset $M \subset \mathbb{P}^{n-1}(\Q)$  is thin if there is an algebraic variety $X$ defined over $\Q$ and a morphism $\pi : X \rightarrow \mathbb{P}^{n-1}$ such that

(i) $M \subset \pi(X(\Q))$, and 

(ii) the fibre of $\pi$ over the generic point is finite and $\pi$ has no rational section over $\Q$.  

\noindent
A thin set $M \subset \mathbb{A}^n(\Q)$ is defined accordingly, with $\mathbb{P}^{n-1}(\Q)$ replaced in each instance by $\mathbb{A}^n(\Q)$.

 This notion was initially recorded in a course of J.-P. Serre in 1980-1981 at the Coll\`ege de France, which we cite here in the form \cite{Ser97}.   Serre's   motivation in that text included the study of irreducibility, and specialization of Galois groups. Indeed, if $F(Y,X_1,\ldots,X_n)$ is an irreducible polynomial over $\Q(X_1,\ldots,X_n)$, there is a thin set $M \subset \mathbb{A}^n$ such that for $(x_1,\ldots,x_n) \not\in M$, $F(Y,x_1,\ldots,x_n)$ is irreducible over $\Q$, and  there is a thin set $M' \subset \mathbb{A}^n$ such that for $(x_1,\ldots,x_n) \not\in M'$, $F(Y,x_1,\ldots,x_n)$ has the same Galois group as $F(Y,X_1,\ldots,X_n)$   over $\Q(X_1,\ldots,X_n)$ \cite[\S 9.2-9.3]{Ser97}. 

  Serre   asked how many points of bounded height can lie in a thin set, and this question has motivated significant research since the 1980's, culminating in a series of recent breakthroughs (including \cite{Sal23} and \cite{BCSSV25x}).   The first goal of the present manuscript is to survey  recent results, within the context of   the four paradigms for thin sets, according to the projective/affine and type I/type II designations. Since some  central questions within the four paradigms have recently been resolved completely, it is a good time to take stock, and also to highlight remaining open questions.

The second  goal of the present manuscript is to consider questions of uniformity: how does an upper bound for the number of points in a thin set $M$ depend on the variety $X$ that characterizes the thin set? Famously, in the setting of type I thin sets, the best-known upper bound for the number of points of bounded height is independent of the maximum size, say $\|F\|$, of the coefficients of the polynomials defining $X$; such an upper bound is called uniform. In the case of type II thin sets, no such uniform upper bound is known. We first recall an  upper bound due to Serre \cite{Ser97}, with unspecified dependence on $\|F\|$, which has served as a baseline for over 40 years.  We   refine the method of Serre  to clarify how the implicit constant may depend on $\|F\|$, in certain special cases. Then, we describe  a different  method to prove the baseline upper bound, first noted in our recent work \cite{BPW25x}; this method applies in full generality and shows there is at most polylogarithmic dependence on $\|F\|$. Finally, we construct counterexamples to a specific type of strong, uniform upper bound for  points in affine thin sets of type II, indicating that the question of uniformity may be more subtle for type II   than for type I thin sets.

\subsection{Thin sets of type I and type II: definitions}\label{sec_types}
In order to  survey recent progress in detail, we now define two types of thin set, which Serre distinguished in \cite[Ch. 9]{Ser97}, and also in \cite[\S 3.1]{Ser92}. 
 We state the definitions first in the projective setting:

\emph{Type I:} A projective thin set $M \subset \mathbb{P}^{n-1}(\Q)$   is of type I if  there is a Zariski-closed subvariety $V \subsetneq \mathbb{P}^{n-1}$   such that $M \subset V(\Q)$.

\emph{Type II:} A  projective thin set $M \subset \mathbb{P}^{n-1}(\Q)$ 
  is of type II
 if there is an irreducible projective  algebraic variety $Z$ over $\Q$ with $\dimP  Z = \dimP  (\mathbb{P}^{n-1})$, and a   dominant morphism $\pi: Z \rightarrow \mathbb{P}^{n-1}$ with generically finite fibres,  of degree $d \geq 2$,  with $M \subset \pi (Z(\Q))$.

Type I and type II affine thin sets $M \subset \mathbb{A}^n(\Q)$ are defined in an analogous way, with $\mathbb{P}^{n-1}$ replaced in each instance by $\mathbb{A}^n$, and projective replaced by affine. (Serre's presentation of thin sets in \cite[Ch. 9, 13]{Ser97} is stated over any number field;  we restrict our attention in the present survey to $\Q$.) With an  analytic number theory audience in mind, we provide a motivation for these two types, and prove the following lemma in  \S \ref{sec_classifying_two types}-\ref{sec_proof_lemma_thin_set_two_types}.
\begin{lem}\label{lemma_thin_set_two_types}
 Every thin set in $\Pbb^{n-1}(\Q)$ (respectively in $\Abb^{n}(\Q)$) is a finite union of thin sets of type I and type II.
 \end{lem}

There is a useful polynomial interpretation of the type I/type II dichotomy, which we review in the affine setting, as in Serre \cite[\S 9.1]{Ser97}. If $M\subset \mathbb{A}^n(\Q)$ is an affine thin set of type I, then there is a nonconstant    polynomial $G \in \Q[X_1,\ldots,X_n]$   such that 
\[ M \subset \{ \bx \in \Q^n: G(x_1,\ldots,x_n)=0\}.\]
On the other hand, given an irreducible  $F(Y,X_{1},...,X_{n})\in\mathbb{Q}(X_{1},...,X_{n})[Y]$,  a polynomial in $Y$   with $\deg_{Y} F\geq 2$,  then the following set   is an affine thin set of type II:
\beq\label{affine_typeII_polynomial_interp}
  \{\bfx\in  \Q^n:  \text{$\bfx$ not a pole of any coefficient of $F$, $F(Y,\bfx)=0$ is solvable over $\Q$}\} \subset \Abb^n(\Q).\eeq
By replacing  $F(Y,\bX)$ by a multiple $\tilde{F}(Y,\bX)$ of $F$ by an appropriate polynomial in $X_1,\ldots,X_n$ so that $\tilde{F}(Y,\bX) \in \Q[Y,X_1,\ldots,X_n]$, the set depicted above is contained in the set 
\[
\{\bx \in \Q^n: \text{$\tilde{F}(Y,\bx)=0$ is solvable over $\Q$}\}.\]
Thus it is no loss of generality   to assume that within (\ref{affine_typeII_polynomial_interp}), $F$ is   a polynomial $Y,X_1,\ldots,X_n$. Moreover, modulo a thin set of type I, every thin set of type II takes the form of (\ref{affine_typeII_polynomial_interp}), in the following sense (as we prove in \S \ref{sec_proof_typeII_poly_interp}, following \cite{Ser97,Bro03N,BCSSV25x}):
 \begin{lem}\label{lemma_typeII_poly_interp}
Let $M \subset \mathbb{A}^{n}(\Q)$ be an affine thin set of type II, with corresponding irreducible affine variety $Z$ over $\Q$ with $\dimA Z=n$, and dominant morphism $\pi: Z \maps \Abb^n$ with generically finite fibres, of degree $d \geq 2$, such that $M \subset \pi(Z(\Q))$. Then there exists a proper closed subset $V \subsetneq Z$ with associated morphism $\pi_V:=\left. \pi \right|_V$, and an irreducible polynomial $F \in \Z[Y,X_1,\ldots,X_n]$  that is monic in $Y$ and has $\deg_YF=d$, such that
\[ M \subset \left(\pi_V (V) \cup \{\bx \in \Q^n: \text{$F(Y,\bx)=0$ is solvable over $\Q$}\} \right).\]
 Moreover, $\pi_V(V)$ is a thin set of type I in $\Abb^n$.
 \end{lem}
We remark that  the  degree of $\pi_V$ is inexplicit; see Remark \ref{remark_poly_interp_consequence}.

\subsection{The counting problems}
Our focus is on counting points of bounded height in a thin set, a suite of problems formulated explicitly by Serre  \cite[\S 13]{Ser97}.
In the projective setting, let  $M \subset \mathbb{P}^{n-1}(\Q)$ be a thin set, and for each  $\bx=[x_1:\ldots:x_{n}] \in \Pbb^{n-1}(\Q)$  represented by $(x_1,\ldots,x_{n}) \in \Z^{n}$ with $\gcd(x_1,\ldots,x_{n})=1$, define 
$\|\bx\|=\max_{1 \leq i \leq n}|x_i|$.
Define the counting function  
\beq\label{thin_counting_function_projective_dfn} N_{\Pbb^{n-1}}(M,B) := \#\{\bx \in M \subset \mathbb{P}^{n-1}(\Q): \gcd(x_1,\ldots,x_n)=1, \|\bx\| \leq B\}.
\eeq
Certainly $N_{\Pbb^{n-1}}(M,B) \ll_n B^n$ is trivially true for all $B \geq 1$.

In the affine setting, suppose  $M \subset \mathbb{A}^n(\Z)  $ is a thin set, and for each integral point $\bx=(x_1,\ldots,x_n) \in \Abb^n(\Z)$ define 
$\|\bx\|=\max_{1 \leq i \leq n}|x_i|$. 
Then define the counting function
\beq\label{thin_counting_function_affine_dfn}
N_{\Abb^{n}}(M,B) := \#\{\bx \in M \subset \mathbb{A}^n(\Z): \|\bx\| \leq B\}.
\eeq
Certainly $N_{\Abb^{n}}(M,B) \ll_n B^n$ is trivially true for all $B \geq 1$. 

 The following nontrivial upper bound was established by Serre in  \cite[Ch. 13 Thm. 1]{Ser97}, and is closely related to work of Cohen \cite[Thm. 2.1]{Coh81}.
\begin{thm}[Baseline upper bound]\label{thm_baseline_affine}
For any thin set $M \subset \mathbb{P}^{n-1}(\Q)$, for some $C(M) \geq 1$ and $\ga(M)>0$,
\beq\label{thin_set_bound_proj_intro}
N_{\Pbb^{n-1}}(M,B) \leq C(M) B^{n-1/2} (\log (B+2))^{\ga(M)}, \qquad \text{for all $B \geq 1$.}
\eeq
For any thin set $M \subset \mathbb{A}^n(\Z)$, for some $C(M) \geq 1$ and $\ga(M)>0$,
\beq\label{thin_set_bound_affine_intro}
N_{\Abb^{n}}(M,B) \leq C(M) B^{n-1/2} (\log (B+2))^{\ga(M)}, \qquad \text{for all $B \geq 1$.}
\eeq
\end{thm}

We will return momentarily to the dependence of $C(M)$ (and the exponent $\ga(M)$) on $M$, which is a significant focus of the present paper. 

A basic observation is in order,  following \cite[p. 187]{Ser97}; we provide a proof in  \S \ref{sec_simple_relation_projective_affine}.
\begin{lem}\label{lemma_affine_version_is_thin} 
Given a thin set $M \subset \mathbb{P}^{n-1}(\Q)$, 
define the set
\[M' = \{\bx \in \Z^{n}\setminus\{\boldsymbol{0}\}: p(\bx) \in M\} \subset \mathbb{A}^n(\Z),\]
in which $p$ maps the tuple $(x_1,\ldots,x_{n})$ into $\mathbb{P}^{n-1}(\Q)$. Then $M'$ is a thin set in $\mathbb{A}^n(\Z)$ and \beq\label{aff_proj_intro}
N_{\Pbb^{n-1}}(M,B) \leq N_{\Abb^{n}}(M',B).
\eeq
Moreover, if $M$ is of type I then $M'$ is a finite union of thin sets of type I; if $M$ is type II then $M'$ is a finite union of thin sets of type I and type II. 
\end{lem}

Thus to establish the baseline upper bound (\ref{thin_set_bound_proj_intro}) for all projective thin sets, it suffices to prove the affine result (\ref{thin_set_bound_affine_intro}).
However, if one aims to improve on the baseline exponent $n-1/2$ in (\ref{thin_set_bound_proj_intro}), this relationship is    of limited use: 
 on the one hand Serre observed that the affine bound (\ref{thin_set_bound_affine_intro}) is essentially sharp (see Example \ref{example_affineII_lower_bound}). On the other hand, Serre asked an influential question about the projective setting:
 \begin{conj}[Thin set, projective]\label{conj_Serre_thin_proj}
 For any   thin set $M \subset \mathbb{P}^{n-1}(\Q)$,  for some $C(M) \geq 1$ and $\ga (M) >0$,
\beq\label{Serre_question_projective}
N_{\Pbb^{n-1}}(M,B) \leq C(M) B^{n-1}(\log (B+2))^{\ga(M)} \qquad \text{for all $B \geq 1$.}
\eeq
\end{conj}
This prediction is stated in both \cite[p. 178 after Thm. 3]{Ser97} and \cite[p. 27]{Ser92}. 
Comparing this conjecture to the fact that (\ref{thin_set_bound_affine_intro}) cannot be improved shows it is essential to consider the affine and projective settings separately---yet they also interact in important ways. In \S \ref{sec_lit_typeI}-\S \ref{sec_lit_typeII}, we present a survey  of results on counting points in thin sets, which clearly separates  out each of the four   combinations of projective/affine and type I/type II designations.  In particular, we remark that Conjecture \ref{conj_Serre_thin_proj} has now largely been resolved by the recent work of \cite{BCSSV25x}, with related progress  in the affine case in work of the authors \cite{BPW25x}.

\subsection{Questions of dependence and uniformity}
In the second half of the paper  we present original results, focused on obtaining refined information  for $C(M)$ and $\gamma(M)$ in Theorem \ref{thm_baseline_affine}, and with a specific focus on affine type II thin sets. 
Given any polynomial $F \in \Z[Y,X_1,\ldots,X_n]$ with $\deg_Y F \geq 2$, define
\beq\label{NFB_intro}
N^{\mathrm{cov}}_{\Abb^{n}}(F,B):=\#\{\bfx\in[-B,B]^n \cap \Z^n:  \text{ $F(Y,\bfx)=0$ is solvable over $\Z$}\}.
\eeq
It is always trivially true that $N^\cov_{\Abb^n}(F,B) \ll B^n$ (and this is sharp in the case that $F(Y,\bX)$ has $\deg_Y F \leq 1$).  
To prove   Theorem \ref{thm_baseline_affine} for an arbitrary affine thin set   (and hence also for an arbitrary projective thin set, by (\ref{aff_proj_intro})), it will suffice to prove for all polynomials  $F$ with $\deg_Y F \geq 2$ that are absolutely irreducible (that is, irreducible over $\overline{\Q}$),
 \beq\label{NFB_nonuniform_intro}
N^\cov_{\Abb^n}(F,B) \leq C(F) B^{n-\frac{1}{2}}  (\log(B+2))^{\ga(F)} \qquad \text{for all $B \geq 1$,}
\eeq
in which $C(F)$ and $\ga(F)$ are positive constants that may depend on $F$. (That this is sufficient is suggested already by Lemma \ref{lemma_typeII_poly_interp}, and we will provide the details in   \S \ref{sec_lit_typeII}.) Let $\|F\|$ denote the maximum absolute value of any coefficient of $F$. We say an upper bound of the form (\ref{NFB_nonuniform_intro}) is \emph{uniform} if $C(F)$ (and $\ga(F)$) depend at most on $n$ and the total degree of $F$, but are independent of $\|F\|$.

   Serre's original treatment  proves (\ref{NFB_nonuniform_intro}) via the large sieve, for any absolutely irreducible polynomial $F$, with   a choice of $\ga(F)<1$ but unspecified dependence of $C(F)$ on $\|F\|$ \cite[Ch. 13, Thm. 1]{Ser97}. In Theorem \ref{thm_Serre}, we prove that in the special case that 
 \beq\label{F_special_structure}
 \text{$F(Y,\bX)$ is   a monic polynomial in $Y^m$ for some $m \geq 2$,}
 \eeq
 Serre's strategy can be refined to prove (\ref{NFB_nonuniform_intro}) with a constant $C(F)$ that exhibits at most \textit{polylogarithmic} dependence on $\|F\|$.  For $m=1$ (so that $F(Y,\bX)$ is only assumed to be monic in $Y$), we observe that this refinement   can be obtained under the assumption of GRH for a single Dedekind zeta function.   (We also remark that contemporaneous with Serre's work, Cohen studied an upper bound for the number of specializations of a polynomial $F(Y,\bX)$ to $F(Y,\bx)$ for which the corresponding Galois group is not preserved \cite[Thm. 2.1]{Coh81}. Modifying Cohen's approach yields an upper bound of the form (\ref{NFB_nonuniform_intro}) with $C(F)$ having at most \emph{polynomial} dependence on $\|F\|$ (over an arbitrary number field). We consider Cohen's strategy in detail in a different work \cite{BPW25x_gen}.)

In our recent work \cite{BPW25x} we require a  stronger result, namely that (\ref{NFB_nonuniform_intro}) holds with $C(F)$ having at most \emph{polylogarithmic} dependence on $\|F\|$, without assuming the special structure (\ref{F_special_structure}) of $F$. In Theorem \ref{thm_polylog} we prove such a bound; this argument originally appeared as \cite[Thm. 2.13]{BPW25x}. This confirms that Theorem \ref{thm_baseline_affine} holds, with $C(M)$ having at most polylog dependence on the coefficients of the polynomial(s) defining the thin set $M$.

Finally, we consider the question of (potential) \emph{uniformity} when bounding $N^\cov_{\Abb^n}(F,B)$, in the setting of affine thin sets of type II. As we will survey in \S \ref{sec_lit_typeI}, the uniformity of estimates for $N_{\Pbb^{n-1}}(M,B)$ and  $N_{\Abb^{n}}(M,B)$ has been extensively studied when $M$ is a thin set of type I. 
In \S \ref{sec_uniformity_question} we raise the question of whether uniformity for  thin sets of type II is more subtle. In particular, we present   counterexamples to a  (putative) uniform upper bound  for affine thin sets of type II, and prove:

     \begin{thm}\label{thm_counterex} 
Let $n\geq 1$ be given. As $k$ varies over positive integers, there is a family of  thin sets $M_k \in \Abb^n(\Z)$  of type II, defined by polynomials $F_k$ with $\|F_k\|=k$,   for which there is no constant $c>0$ such that 
\[N_{\Abb^{n}} (M_k, B) \ll_{n,\deg F_k}B^{n-1}(\log B)^{c}\]
can hold as $k \rightarrow \infty$, with an implicit constant independent of $k$.
\end{thm}

\subsection{Notation}
We use $A \ll B$ to denote that $|A| \leq C B$ for some constant $C>0$, where $B$ is non-negative. We use $A\ll_\kappa B$ to denote that the constant $C$ may depend upon the parameter $\kappa$. 
We record the five main counting functions extensively studied in this paper.
\begin{itemize}
\item $ N_{\Pbb^{n-1}}(M,B) := \#\{\bx \in M \subset \mathbb{P}^{n-1}(\Q):  \bx=(x_1,\ldots,x_n) \in \Z^n, \gcd(x_1,\ldots,x_n)=1, \|\bx\| \leq B\},$ for a thin set $M \subset \Pbb^{n-1}$.
\item 
$N_{\Abb^{n}}(M,B) := \#\{\bx \in M \subset \mathbb{A}^n(\Z): \bx=(x_1,\ldots,x_n) \in \Z^n, \|\bx\| \leq B\},$  for a thin set $M \subset \Abb^n$.
\item  
$ N_{\proj}(V,B) := \#\{ \bx \in V\cap \Pbb^{N}(\Q):  \bx=(x_1,\ldots,x_{N+1}) \in \Z^{N+1},\gcd(x_1,\ldots,x_{N+1})=1, \|\bx\| \leq B\}$, for an irreducible projective variety $V \subset \Pbb^N$.
\item 
$ N_{\aff}(V,B) := \#\{ \bx \in V\cap \Abb^{N}(\Z):   \bx=(x_1,\ldots,x_N) \in \Z^N, \|\bx\| \leq B\},$ for an irreducible affine variety $V \subset \Abb^N$.
    \item $N^\cov_{\Abb^n}(F,B):=\{\bfx\in [-B,B]^n\cap \Z^n:  \text{$F(Y,\bfx)=0$ is solvable over $\Z$}\}$, for a polynomial $F(Y,X_1,\ldots,X_n)\in \Z[Y,X_1,\ldots,X_n]$, relevant to affine thin sets of type II.
\end{itemize}
We may generally assume $B \geq 3$ so that $\log B \gg 1$.

 \section{Preliminaries: motivation for thin sets of type I and type II}\label{sec_types_motivation} 

In this section we prove Lemmas \ref{lemma_thin_set_two_types}, \ref{lemma_typeII_poly_interp}, and \ref{lemma_affine_version_is_thin}.
The underlying construction leading to the two types of thin sets is intrinsically geometric, and thus we also provide a brief motivation of these ideas, with analytic number theorists in mind. We consider here the projective setting, but analogous arguments motivate the affine setting.

\subsection{Terminology}\label{sec_morphism_terminology}
 We start by recalling some terminology. We define a variety $X$ over a field $k$ to be a (reduced) scheme $X$   such that the structure morphism $X\rightarrow \text{Spec} (k)$ is separated and of finite type (see \cite[\href{https://stacks.math.columbia.edu/tag/01T0}{Section $29.15$}]{StaPro} for the definition of morphism of finite type, and \cite[\href{https://stacks.math.columbia.edu/tag/01KH}{Section 26.21}]{StaPro} for the definition of separated morphism).  
For varieties $X,Y$, a morphism $\pi : X \rightarrow   Y$ is said to be dominant if the image $\pi(X)$ is dense in $Y$. A morphism is said to be generically finite if the generic fibre is finite. 
 In the setting of thin sets, we now confirm the following useful fact:
 \begin{lem} If $X$ is an irreducible variety  and a morphism $\pi: X \maps \Pbb^{n-1}$ is given, the morphism $\pi$ is of finite type (and locally of finite type). 
 \end{lem}
 \begin{proof}
To record the relevant  citations to verify this, consider in general    two integral schemes $X,Y$ of finite type over $\Q$ (and consequently locally of finite type over $\Q$) and let $f: X\rightarrow Y$ be a   morphism (in our case, $f=\pi$ and $Y = \bP^{n-1}$). In particular, since $X$ is locally of finite type over $\mathbb{Q}$, then a  morphism $f:X\rightarrow Y$ is locally of finite type, by Lemma 29.15.8  in \cite[\href{https://stacks.math.columbia.edu/tag/01T0}{Section 01T0}]{StaPro}. 
By   Dfn. 29.15.1 (ibid), to conclude that $f$ is of finite type, it suffices to show that in addition to being locally of finite type,  $f$ is quasi-compact. 
(As in Dfn. 5.12.1 \cite[\href{https://stacks.math.columbia.edu/tag/005A}{Definition 005A}]{StaPro}, a topological space $V$ is quasi-compact if every open covering of $V$ has a finite subcover. A continuous map $f:X\maps Y$ is quasi-compact if the inverse image $f^{-1}(V)$ of every quasi-compact open $V \subset Y$ is quasi-compact.) Thus suppose $V \subset Y$ is quasi-compact, and consider for the given morphism $f: X \maps Y$ the inverse image $f^{-1}(V)$. Since $X$ is a noetherian scheme ($X\rightarrow \text{Spec} (k)$ is of finite type and hence $X$ is noetherian \cite[\href{https://stacks.math.columbia.edu/tag/01T0}{Lemma $29.15.6$}]{StaPro}), then $f$ is quasi-compact thanks to \cite[\href{https://stacks.math.columbia.edu/tag/01OU}{Lemma $28.5.8$}]{StaPro}.
\end{proof}

 \subsection{The two types} \label{sec_classifying_two types}
Recall the definition:
A  subset $M \subset \mathbb{P}^{n-1}(\Q)$  is thin if there is an algebraic variety $X$ over $\Q$ and a morphism $\pi : X \rightarrow \mathbb{P}^{n-1}$   with the properties

(i) $M \subset \pi(X(\Q))$, and 

(ii) the fibre of $\pi$ over the generic point is finite and $\pi$ has no rational section over $\Q$.  
\\
Let us see how this   leads to the distinction between type I and type II thin sets, as defined in \S \ref{sec_types}, and the proof of Lemma \ref{lemma_thin_set_two_types}.

Let  $M \subset \Pbb^{n-1}(\Q)$ be a thin set,  with corresponding   variety $X$ defined over $\Q$ and   morphism $\pi : X \rightarrow \Pbb^{n-1}$. We furthermore assume $X$ is irreducible and reduced. (This is sufficient for studying the general case, after  decomposing $X$ into irreducible components.) 
 There are naturally two cases to consider: 

(I) $\pi: X \maps \Pbb^{n-1}$ is not dominant; 

(II) $\pi: X \maps \Pbb^{n-1}$ is dominant. 

Case (I): In the first case, the image $\pi(X)$ is not dense in $\Pbb^{n-1}$, that is to say the closure $\overline{\pi(X)}$ is a proper closed subset of  $\Pbb^{n-1}$. Since $M \subset \pi(X(\Q))$ by property (i) of the definition, it follows that  $M$ is contained within the rational points $V(\Q)$ of a proper closed subvariety $V \subsetneq \Pbb^{n-1}$. This leads to the notion of a thin set of type I.  

The second case, when $\pi$ is dominant, is more subtle. (Indeed, Serre   called thin sets of type II the ``most interesting case of a thin set'' \cite[p. 121]{Ser97}.)

We now record two lemmas:

\begin{lem}[{\cite[Corollary $13.1.6$]{EGAIV_part3}}]\label{lemma_fibre_dim}
  If   $f: X \maps Y$ is a dominant morphism that is locally of finite type, and $r=\dim X - \dim Y$, then every irreducible component of every fibre $f^{-1}(y)$ has dimension at least $r$. In particular, if $f$ has generically finite fibres, then $r=0$.
\end{lem}

\begin{lem}\label{lemma_dominant_degree1_rational-section}
 Let $X,Y$ be irreducible varieties over a field $k$, and consider a dominant morphism $f: X \maps Y$ of finite type with generically finite fibres. Then $f$ has  degree $1$ if and only if $f$ has a rational section over $k$.
\end{lem}
  We defer for now the proof of Lemma $\ref{lemma_dominant_degree1_rational-section}$, and see   how it applies to the case (II).

Case (II):  By Lemma \ref{lemma_fibre_dim}, since $\pi: X \maps \Pbb^{n-1}$ is a   dominant morphism of finite type with generically finite fibres, it follows that $\dimP X = \dimP \Pbb^{n-1}$. Next, by Lemma \ref{lemma_dominant_degree1_rational-section},   $\pi: X \maps \Pbb^{n-1}$ has degree 1 if and only if $\pi$ has a rational section over $\Q$. By property (ii) in the definition of thin set, $\pi$ does not have a rational section over $\Q$, and thus it must be that $\pi$ has degree at least 2. Thus case (II) leads directly to the notion of a thin set of type II.

In preparation for proving Lemma \ref{lemma_dominant_degree1_rational-section} it is useful to set some terminology. For varieties $X,Y$, a rational map $\pi: X \ratlmaps Y$ is a morphism defined on an open dense set in $X$. (More precisely, it is an equivalence class of pairs $\langle U, \phi_U \rangle$ where $U$ is a non-empty open subset of $X$ and $\phi_U$ is a morphism from $U$ to $Y$, and $\langle U, \phi_U \rangle$ is equivalent to $\langle V, \phi_V \rangle$ if $\phi_U$ and $\phi_V$ agree on $U \cap V$; see \cite[Ch. I Section 4]{Har77}, from which the definitions in this discussion are taken.) A section of a morphism $f: X \maps Y$ is a morphism $s: Y \maps X$ such that $f\circ s$ is the identity morphism on $Y$. A rational section $s$ of $f: X \maps Y$ is a rational map $s: Y \ratlmaps X$ such that $f \circ s = \mathrm{id}_Y$, as a rational map (i.e. on an open dense subset of $Y$). 
To consider the notion of the degree of a morhpism $f$, we recall another fact:
\begin{lem} \label{lemma_induced_injective}
For   varieties $X,Y$ over a field $k$,  a generically finite dominant morphism $f: X \maps Y$ yields an induced map 
\beq\label{induced_map_injective} 
f^* : k(Y) \rightarrow k(X)
\eeq
on the corresponding function fields, defined by pulling back rational functions from $Y$ to $X$; that is to say, acting by $f^*: r \mapsto r\circ f$. 
The induced map $f^*$ is  injective, under the hypothesis that $f$ is dominant. 
\end{lem}
\begin{proof} Suppose that a rational function $r \in k(Y)$ (defined on a dense open set $U_Y \subset Y$) lies in the kernel of $f^*$, so that $r \circ f = 0$.  Without loss of generality we may assume that $r=P/Q$, for some $P,Q\in\mathbb{C}[\bfX]$ on $U_{Y}$; note also that since $U_Y$ is dense and open then its complement $Y \setminus U_Y$ is nowhere dense in $Y$. Since $r(f(x))=0$ for every $x \in X$ such that $f(x) \in U_Y$, it follows that upon defining $V:= U_Y \cap f(X)$, the restriction $\left. r\right|_{V}=0=\left. P\right|_{V}$.   
Now under the  hypothesis that $f:X \maps Y$ is dominant, the image $f(X)$ is dense in $Y$, and this implies that the set $U_{Y}\cap f(X)$ is also dense in $Y$. (For upon writing $f(X)= (U_{Y}\cap f(X))\cup  ((Y\setminus U_{Y})\cap f(X))$, the fact that $((Y\setminus U_{Y})\cap f(X))$ is nowhere dense (as a subset of the nowhere dense set $Y \setminus U_Y$)  forces $U_Y \cap f(X)$ to be dense.) Hence $P$ is zero on the dense set $U_{Y}\cap f(X)$ in $Y$, so it must be the zero polynomial, and consequently  $\left. P\right|_{U_{Y}}=0$. This implies  $\left. r\right|_{U_{Y}}=0$ on the dense open set $U_Y$,  so $r$ lies in the same equivalence class as the zero map on $Y$; that is, we may say $r=0$ (see Dfn. 29.49.1 in   \cite[\href{https://stacks.math.columbia.edu/tag/01RR}{Section 01RR}]{StaPro}). Thus $f^*$ is injective, as claimed.  
\end{proof}

With the knowledge that the induced map is injective, the degree of $f$ is then defined to be 
\[ \deg f:= [k(X): f^*k(Y)],\]
the degree of the function field $k(X)$ as a (finite) extension of fields of $f^*k(Y)$.  Notice that the fact that $f$ is a dominant, generically finite, morphism of finite type and $X,Y$ are integral schemes guarantees that $\deg f<\infty$. Indeed, since $f:X\rightarrow Y$ is generically finite, we can find two affine open subsets $U\subset X$, and $V\subset Y$ such that $\left.f\right|_{U}: U\rightarrow V$ is finite. Since $f_{|U}$ is finite and dominant ($U$ and $V$ are dense in $X$ and $Y$ respectively and $X,Y$ are irreducible), we have that $[k(U):k(V)]<\infty$. On the other hand since $U,V$ are dense in the integral schemes $X,Y$, one has $k(U)=k(X)$, and $k(V)=k(Y).$ See \cite[Chapter II Exercise 3.7]{Har77}, and also Dfn. 29.51.8 in \cite[\href{https://stacks.math.columbia.edu/tag/02NV}{Section 02NV}]{StaPro}. 

Next,  we consider in particular the case of degree 1:
\begin{lem}\label{lemma_birational}
    Under the hypotheses of Lemma \ref{lemma_induced_injective}, $f:X \maps Y$ has degree 1 if and only if $X$ and $Y$ are birationally equivalent.
\end{lem}
To parse this, recall that a birational map $\phi: X \ratlmaps Y$ is a rational map with an inverse   $\psi: Y \ratlmaps X$ such that $\psi \circ \phi = \mathrm{id}_X$ as a rational map, and $\phi \circ \psi = \mathrm{id}_Y$ as a rational map. If such a birational map exists, $X$ and $Y$ are said to be birationally equivalent. 
\begin{proof}
By definition, $f$ has degree 1 precisely when  $[k(X): f^*k(Y)]=1$, so that (equivalently) the induced map $f^*: k(Y) \maps k(X)$ is not only injective but also surjective; that is to say, the function fields $k(X)$ and $k(Y)$ are isomorphic, with isomorphism $f^*$. 
It is a classical result that $k(X)$ and $k(Y)$ are isomorphic   precisely when $X$ and $Y$ are birationally equivalent \cite[Ch. I Cor. 4.5]{Har77} (see also \cite[Ch. 29 Dfn. 49.11]{StaPro} and subsequent remark). 
\end{proof}

\begin{proof}[Proof of Lemma \ref{lemma_dominant_degree1_rational-section}]
Now to prove Lemma \ref{lemma_dominant_degree1_rational-section}, suppose that a dominant morphism $f:X \maps Y$ of finite type with generically finite fibre has degree 1, so that by the discussion above, $X$ and $Y$ are birationally equivalent, with isomorphism $f^*: k(Y) \maps k(X)$.
  After passing to an affine cover for $X$ and $Y$, we may proceed in the case that $X$ and $Y$ are affine. Then by \cite[Ch. I Prop. 3.5]{Har77} there is a bijection between $\mathrm{Hom}(X,Y)$ and $\mathrm{Hom}(A(Y),A(X))$ (in which $A(X)$ denotes the coordinate ring, with $k(X)$ isomorphic to the quotient field of $A(X)$, and analogously for $Y$ \cite[Ch. I Thm. 3.2]{Har77}). Thus there exists a morphism $g: Y\rightarrow X$ such that $g^*=(f^{*})^{-1}$. Now take $f\circ g$; then since $(f\circ g)^{*}= g^{*} \circ f^*=\Id_{k(Y)}$, then by a second application of \cite[Ch. I Prop. 3.5]{Har77}, we conclude that $f\circ g=\Id_{Y}$. Since a morphism is a rational map (defined on $Y$ itself), this suffices to show $f$ has a rational section.
 
In the other direction, suppose that a dominant morphism $f: X \maps Y$ has a rational section, so that there is a rational map $s: Y \ratlmaps X$ (defined on a dense open subset $U_Y$ of $Y$) such that $f \circ s =\mathrm{id}_Y$ as a rational map.    Then by \cite[Ch. I Prop. 3.5]{Har77}, 
\beq\label{must_be_surjective} 
(f\circ s)^* =  s^*  \circ f^*=\mathrm{Id}_{k(Y)}.
\eeq 
Since $f^*: k(Y) \maps k(X)$ and $s^* : k(X) \maps k(Y)$,   the property (\ref{must_be_surjective}) can only be achieved if $s^*$ is surjective. A  ring homomorphism acting on a field (here $k(X)$)  is either injective or the zero map, and since $s^*$ is surjective it is not the zero map, so $s^*$ is injective as well, and hence provides an isomorphism so that   $k(X)\cong k(Y)$. Hence $\deg [k(X):k(Y)]=1$, and finally this implies $\deg f=1$, as desired.
\end{proof}

We close this section  with a few final remarks.  Recall the notation of \S \ref{sec_types}.
For clarity, let us confirm that a set that satisfies the condition of being type I is a thin set:  one may take $X=V$ and $\pi: X \maps \Pbb^{n-1}$ to be the inclusion map, for which the generic fibre is empty since $V$ is a proper closed subset.  Similarly, let us confirm that a set of type II is a thin set:  one may take $X=Z$ and $\pi:Z\rightarrow \bP^{n-1}$ the dominant morphism with generically finite fibres; recall that by hypothesis $\dimP Z = \dimP(\bP^{n-1})$. 
Since $\deg(\pi) \geq 2$, $\pi$ has no rational section over $\Q$, by Lemma \ref{lemma_dominant_degree1_rational-section}.
Finally, the space of thin sets is closed under finite union, that is, any finite union of thin sets is thin.

\subsection{Proof of Lemma \ref{lemma_thin_set_two_types}}\label{sec_proof_lemma_thin_set_two_types}

Given  a thin set $M$ with corresponding variety $X$ and morphism $\pi: X \maps \Pbb^{n-1}$ such that $M \subset \pi(X(\Q))$, decompose $X=\cup_{i=1}^{m}V_{i}$ into irreducible components. Then for every $i=1,...,m$ we have a morphism $\pi_i = \left.\pi\right|_{V_{i}}:V_{i}\rightarrow \mathbb{P}^{n-1}$, and upon setting $M_i = M \cap \pi_i (V_i(\Q))$ then $M = \cup_{i=1}^n M_i$. The discussion in \S \ref{sec_classifying_two types} shows that if $\left.\pi\right|_{V_{i}}$ is dominant then $M_i$ is  a thin set of type II, while if not, then $M_i$ is a thin set of type I. This completes the proof of Lemma \ref{lemma_thin_set_two_types}.

\subsection{Proof of Lemma \ref{lemma_typeII_poly_interp}}\label{sec_proof_typeII_poly_interp}

To prove the lemma, we adapt portions from the proof of \cite[Lemma 3]{Bro03N}, following a remark of Serre \cite[p. 122]{Ser97}; see also Proposition \ref{prop_BCSSV_affineII_projII}, which is due to \cite{BCSSV25x}.
Define $U\subset Z$ to be the affine variety $\pi^{-1}(\Abb^n)$, and let $k(U)$ denote the function field of $U$. (Note for later reference that $Z\setminus U$ is a proper closed subset of $Z$. Certainly $Z\setminus U$ is a proper subset of $Z$, since otherwise $U$ would be empty,   contradicting the fact that $\pi: Z \maps \Abb^n$ is dominant. Additionally, $Z\setminus U$ is closed in $Z$, since the pre-image $U$ of the open set $\Abb^n$ is open, under the morphism $\pi$.) Let $t$ be a generator of $k(U)$ over $k(\Abb^n)=\Q(X_1,\ldots,X_n)$, chosen so that its minimal polynomial, say 
\[ F(Y,X_1,\ldots,X_n)=Y^d + f_1(\bX)Y^{d-1} + \cdots + f_d(\bX),\] 
has coefficients $f_j(\bX)$ that are polynomials with integral coefficients. Now consider the affine variety $W \subset \Abb^{n+1}$ defined by $F(Y,\bX)=0$, as well as the projection map $p: W \maps \Abb^n$ given by $p(Y,\bX)=\bX$. By construction, the function fields $k(U)$ and $k(W)$ are isomorphic (indeed, they are both isomorphic to  $\Q(X_1,\ldots,X_n)(t)$). Consequently, by Lemma \ref{lemma_birational}, $U$ and $W$ are birationally equivalent, and in particular there is a  (bi)rational map $r: U \ratlmaps W$, with an accompanying (bi)rational map $s : W \ratlmaps U$ such that $s \circ r =\mathrm{id}_U$ and $r \circ s= \mathrm{id}_W$ (in each case, as a rational map, that is, on an open dense subset).

From this, observe that there is a proper closed subset $U_0 \subset U$ such that 
\[\pi':=\left. \pi \right|_{U \setminus U_0}: U \setminus U_0 \maps \Abb^n\] 
factors through $p: W \maps \Abb^n$. Indeed, by the existence of the birational map $r: U \ratlmaps W$ there is some open dense subset of $U$ upon which $r$ is a morphism; we define $U_0$ to be the complement of that subset. Then $\pi' := \left. \pi \right|_{U \setminus U_0}= p \circ r$, as in the following diagram:
 \[
  \begin{tikzcd}
    U \setminus U_0  \arrow{dr}{\pi'}  \arrow[dashrightarrow]{d}{r}  \\
   W \arrow[rightarrow]{r}{p} & \mathbb{A}^n
  \end{tikzcd}
\]

Finally, define $V \subset Z$ by $V=U_0 \cup (Z \setminus U)$. First note that $V$ is a proper closed subset of $Z$, since as observed earlier, each of $U_0$ and $Z\setminus U$ is a proper closed subset of $Z$.
Second, by arguing as above, note that $\pi'':= \left. \pi \right|_{Z \setminus V} : Z \setminus V \maps \Abb^n$ factors through $p: W \maps \Abb^n$. That is to say, on $Z \setminus V$, $\pi'' = p \circ s$ for some morphism $s: Z \setminus V \maps W$.

By applying $\pi$ to $Z=V \cup (Z \setminus V)$, we then see that
\[
   \pi(Z)  \subset (\left. \pi  \right|_V(V) \cup \left. \pi\right|_{Z\setminus V}(Z \setminus V) )\\
 = (\left. \pi\right|_V(V) \cup (p \circ s) (Z \setminus V) ) \subset  (\left.\pi\right|_V(V) \cup p (W)) ).
\]
In particular, 
\[ M \subset \pi(Z(\Q))  \subset  (\left. \pi\right|_V(V(\Q)) \cup \{\bx \in \Q^n: \text{$F(Y,\bx)=0$ is solvable over $\Q$}\}  ).
\]
The last step is to verify:
  \begin{lem}\label{lemma_piV_thin_set} 
  If $\pi: Z \maps \Abb^n$ is a dominant morphism with generically finite fibres of degree $d \geq 2$ and $V \subsetneq Z$ is a proper closed subset, then  $\left. \pi \right|_V(V)$ is a thin set of type I in $\Abb^n$.
  \end{lem} 
  \begin{proof} 
  It suffices to show $\left. \pi \right|_V(V)$ is contained in a proper closed subset of $\Abb^n$, which we will do in two steps. First we will confirm that $\dimA \left. \pi \right|_V(V) \leq n-1$ in $\Abb^n$. Second, any set $W \subset \Abb^n$ with dimension $<n$ is contained in a proper closed subset. 

 For the first step, we note that there exists an open subset $U\subset V$ on which $\left.\pi\right|_V$ is finite, since $\pi$ has generically finite fibers. As $\left.\pi\right|_U$ is finite, by Lemma \ref{lemma_fibre_dim} we see $\left.\pi\right|_U(U) = \dimA U = \dimA V \leq n-1$. Additionally, $\left.\pi\right|_U(U)$ is an open subset of $\left.\pi\right|_V (V)$ (with respect to the topology on $\left.\pi\right|_V (V)$), so $\dimA \left.\pi\right|_U(U) = \dimA \left.\pi\right|_V(V)\leq n-1$. 

For the second step, let us verify that  if $W \subset  \Abb^n$ has $\dimA W < n$, then $W$ is contained in a proper closed subset of $\Abb^n$. For suppose on the contrary that $W$ is not contained in any proper closed subset. Then for every (nonempty) open set $U \subset \Abb^n$, $W\not\subset U^c$, so $W \cap U$ is nonempty. Thus $W$ is dense in $\Abb^n$, so that $\dimA W =\dimA \overline{W} = \dimA \Abb^n$, a contradiction.
\end{proof}

\begin{rem}[Inexplicit aspect of Lemma \ref{lemma_typeII_poly_interp}]\label{remark_poly_interp_consequence}

Lemma \ref{lemma_piV_thin_set} confirms that $\left. \pi \right|_V(V)$ is a thin set of type I in $\Abb^n$, but it provides no bound on the degree of this thin set. This inexplicitness presents a barrier to obtaining a uniform upper bound when counting integral points in the thin set of type II considered in Lemma \ref{lemma_typeII_poly_interp}, a subject considered extensively later in the paper. 
  \end{rem}

\subsection{Proof of Lemma \ref{lemma_affine_version_is_thin}: simple relation between projective and affine thin sets}\label{sec_simple_relation_projective_affine}
Let $M$ be the given thin set in $\Pbb^{n-1}$.
    Since $M$ is a finite union of   type I and type II thin sets in $\Pbb^{n-1}$, it suffices to consider those cases individually. 
    
    \subsubsection{Type I} Suppose $M\subset \Pbb^{n-1}$ is of type I, so that for some Zariski-closed (irreducible) variety $V \subsetneq \mathbb{P}^{n-1}$, $M \subset V(\Q)$. 
    Denote homogeneous coordinates in $\Pbb^{n-1}$ by $[x_1: \cdots : x_n]$. For each $i$ let $U_i$ be the open subset $\Pbb^{n-1} \setminus \{x_i=0\}$. Then  for each $i$, $U_i$ can be identified with $\Abb^{n-1}$ by 
    \[ [x_1: \cdots :x_i: \cdots : x_n] \mapsto (x_1/x_i, \ldots,x_n/x_i),\]
omitting the $i$th place in the image.   Then  $V_i:=V \cap U_i$ is an open subset of $V$ and can be identified with an affine variety  $\widetilde{V}_i \subset \Abb^{n-1}$ (which has $V$ as its projective closure).  This affine variety is closed, and is irreducible since its projective closure $V$ is irreducible   (see \cite[
\S4.5 Ex. 1]{Sha13} or \cite[\S5.5 Prop.]{Rei88}).

 On the other hand, $\Abb^n \setminus \{\boldsymbol{0}\}$ can be written as a  union $\cup_{i=1}^n A_i$ in which $A_i = \Abb^n \setminus \{x_i=0\}$.  Each set $A_i$ can be identified with $\Pbb^{n-1}$ by 
 \[p_i:(x_1,\ldots,x_i,\ldots,x_n) \mapsto [x_1: \cdots :x_i: \cdots : x_n] = [x_1/x_i: \cdots :1: \cdots : x_n/x_i].\]
 In the definition of $M'$ in the lemma, we will regard the map $p(\bx)$ as acting by $p_i(\bx)$ for $\bx \in A_i$.  For $M'$ as in the lemma,
write $M' = \cup_{i=1}^n M_i' $ with $M_i = M' \cap A_i$. 
Then for any $\bx \in M_i'$, $p(\bx) \in U_i$ and $p(\bx) \in M \subset V$, so that $p(\bx) \in V_i = V\cap U_i$. 
Consequently, $M_i' \subset \widetilde{V}_i (\Z) \subset \Abb^{n-1}(\Z) \subsetneq \A^n(\Z)$. Thus $M_i'$ is a thin set of type I. 
In conclusion $M'$ is a finite union of thin sets of type I.   

\subsubsection{Type II}
Next, suppose $M\subset \Pbb^{n-1}$ is of type II, so there is an irreducible projective  algebraic variety $Z \subset \P^N$ over $\Q$ with $\dimP  Z = \dimP  (\mathbb{P}^{n-1})$, and a dominant morphism $\pi: Z \rightarrow \mathbb{P}^{n-1}$ with generically finite fibres,  of degree $d \geq 2$, with $M \subset \pi (Z(\Q))$. 
Similar to the argument above, define for each $j=1,\ldots,N$ the affine cover $Z_j=Z \cap U_j$ for $U_j = \Pbb^N \setminus \{x_j=0\}$; then $Z_j$ is an open dense subset of $Z$. Consider the morphism $\pi_j : Z_j \maps \Pbb^{n-1}$ defined by $\pi_j=\left. \pi \right|_{Z_j}$. As in the discussion of \S \ref{sec_classifying_two types}, there are two cases: if $\pi_j$ is not dominant, then $\pi_j(Z_j)$ is a thin set of type I, so the previous case applies.

On the other hand, if $\pi_j$ is dominant, we claim $\pi_j(Z_j)$ is a thin set of type II. This last observation merely requires checking that $\pi_j$ has degree at least 2. By Lemma \ref{lemma_dominant_degree1_rational-section}, it suffices to show that if $\pi_j$ has a rational section then $\pi$ has a rational section. Suppose $\pi_j$ has a rational section, say a rational map $s: \Pbb^{n-1} \ratlmaps Z_j$ defined on a dense open set $O \subset \Pbb^{n-1}$ such that $\pi_j \circ s = \mathrm{id}_{\Pbb^{n-1}}$ as a rational map, or in other words $\pi_j \circ s - \mathrm{id}_{\Pbb^{n-1}}$ is the zero map, as a rational map. Since $Z_j$ is open and dense in $Z$, arguing as in the proof of Lemma \ref{lemma_induced_injective} shows that   $\pi \circ s = \mathrm{id}_{\Pbb^{n-1}}$. Thus $s$ is a rational section of $\pi$, and the claim is verified.

Finally, define for each $j=1,\ldots,N$, and $i=1,\ldots,n$,
\[ M'_{i,j} = \{ \bx \in A_i: p(\bx) \in  \pi_j(Z_j(\Q))\}.\]
Then for the affine set $M'$ as defined in the lemma,
\[ M' \subset \bigcup_{i,j} M'_{i,j}.\]
Each set $M_{i,j}' \subset \A^n$ is a thin set of type I or type II, so that $M'$ is a union of thin sets, and hence is a thin set, and the lemma is proved.  
 
\begin{rem} The elementary ideas of  Lemma \ref{lemma_affine_version_is_thin}
 can be compared to Proposition  \ref{prop_BHBS06_affineI_projI}, which is a  method for passing from a projective thin set of type I to a particular kind of affine thin set of type I, and to Proposition \ref{prop_BCSSV_affineII_projII}, which is a  method  of passing from a projective thin set of type II to the union of a particular kind of affine thin set of type II, and a thin set of type I.
\end{rem}

\section{Survey of  type I thin sets and Dimension Growth}\label{sec_lit_typeI}
In this section we survey results on counting points in thin sets of type I.
 \subsection{Type I, projective}
Consider a thin set $M \subset \Pbb^{n-1}(\Q)$ of type I in the projective setting. Let $V$ be a Zariski-closed subvariety $V \subsetneq \Pbb^{n-1}$ such that $M \subset V(\Q)$, so that
\beq\label{projI_DGC_relation}
N_{\Pbb^{n-1}}(M,B) \leq \#\{ \bx \in V (\Q): \bx \in \Z^n, \gcd(x_1,\ldots,x_n)=1, \|\bx\| \leq B\}.
\eeq
Thus in the type I setting, the central aim is to bound, for any irreducible projective variety $V \subset \Pbb^N$,  
\[ N_{\proj}(V,B) := \#\{ \bx \in V\cap \Pbb^{N}(\Q): \bx \in \Z^n, \gcd(x_1,\ldots,x_{N+1})=1, \|\bx\| \leq B\}.\]
We record a trivial bound:
 
\begin{lem}[Trivial bound, projective variety]\label{lemma_projI_trivial_bound}
Given a  projective variety $V \subset \Pbb^{N}$, then \[
N_{\proj}(V,B) \ll_{\deg V,\dimP V,N} B^{\dimP V+1}.\]
\end{lem}

This is best possible if $V$ is of degree 1. It is simple to show a non-uniform bound $N_{\proj}(V,B) \ll_{V} B^{\dimP V+1}$. Indeed, upon letting $m=\dimP V$, this bound follows from choosing homogeneous coordinates for $\Pbb^N$ so that projection onto the first $m+1$ coordinates induces a finite morphism $\pi: V \rightarrow \Pbb^m.$ Each fibre of $\pi$ has a finite  number of points (controlled by the degree of $\pi$), and then it suffices to count trivially the points in $\Pbb^m$ of height at most $B$; see remarks following \cite[Ch. 13, Thm. 2]{Ser97}.
A different, but still elementary, argument leads to the uniform bound in Lemma \ref{lemma_projI_trivial_bound}. It relies on a trivial bound in the affine setting (Lemma \ref{lemma_affineI_trivial_bound}, proved independently in a section on affine results).

\begin{proof}
Arguing as in the proof of Lemma \ref{lemma_affine_version_is_thin} in \S \ref{sec_simple_relation_projective_affine} (or as in \cite[Ch. I Cor 2.3]{Har77}), for $V \subsetneq \Pbb^{N}$ of projective dimension $m$, we can define an affine variety 
\[ V' = \{\bx : p(\bx) \in V\} \subset \Abb^{N+1},\]
with
$V' = \cup_{i} V_i'=\cup_i (V' \cap A_i)$, in which $A_i = \Abb^{N+1} \setminus \{x_i=0\}$ for $i=1,\ldots,N+1$.
Decompose $N_\proj (V,B) \leq \sum_i  N_\aff(V_i',B)$. Each   $V_i'$ has affine dimension $ \leq m+1$. A trivial bound in the affine setting then shows   that $N_\aff(V_i',B)\ll_{\deg V_i', \dimA V_i', N} B^{m+1}$ (see Lemma \ref{lemma_affineI_trivial_bound}). This suffices for Lemma \ref{lemma_projI_trivial_bound}. 
\end{proof}
\begin{rem}\label{remark_Dehennin}
In particular, if $V$ is of pure dimension then the dependence in this lemma is at most linear in $\deg V$, following Lemma \ref{lemma_affineI_trivial_bound}.  Dehennin has recently conjectured that the dependence on degree in this projective setting  should be at most sublinear  \cite[Conj. 1.3]{Deh26x}.
\end{rem}

 In   the setting of (\ref{projI_DGC_relation}), $V \subsetneq \Pbb^{n-1}$ has $\dimP V \leq n-2$, and Lemma \ref{lemma_projI_trivial_bound} implies
 \beq\label{projI_trivial_thin_set}
 N_{\Pbb^{n-1}}(M,B) \leq N_\proj(V,B) \ll_{\deg V, n} B^{n-1}.
 \eeq
Thus certainly Conjecture \ref{conj_Serre_thin_proj} is met for projective thin sets of type I. Instead, in this case the relevant deep question is:
 
\begin{conj}[Dimension Growth Conjecture, projective]\label{conj_DGC}
For any irreducible projective variety $V \subset \Pbb^N$ of degree $d \geq 2$ defined over $\Q$,
\beq\label{DGC}
N_{\proj}(V,B) \leq C(V) B^{\dimP V}s(B), \qquad \text{for all $B \geq 1$,}
\eeq
in which $C(V)$ is a constant possibly depending on $V$, and $s(B)$ is a ``small factor,'' which could take the form  $s(B)\ll_\ep B^\ep$ (for any $\ep>0$); or $s(B) \leq (\log B)^{\ga(V)}$ for some $\ga(V) > 0$; or $s(B) \leq 1$.
\end{conj}
The main exponent is essentially the best   that could be expected. For suppose that $V$ contains a rational linear subspace $V'$ of codimension $1$ in $V$; then 
 \[N_\proj (V,B) \geq N_\proj(V',B) \gg B^{\dimP V'+1} = B^{\dimP V}.\]
 Serre provides an example for  which a logarithmic small factor necessarily arises in (\ref{DGC}) \cite[p. 178]{Ser97}, as follows:
 \begin{example}\label{example_quadric}
 Consider the quadric $X_1X_2=X_3X_4$ in $\Pbb^3$. Denoting the divisor function $d(\cdot)$,
\[N_\proj (V,B) \gg \sum_{ z \leq B^2} d(z) \gg B^2 \log B.\]
\end{example}

 A refined conjecture asks for uniformity:
\begin{conj}[Uniform Dimension Growth Conjecture, projective]\label{conj_UDGC}
The Dimension Growth Conjecture \ref{conj_DGC} holds with a choice of   $C(V)$ (and $\ga(V)$, if relevant) dependent on the degree and dimension of $V$ but not otherwise on $V$.
\end{conj}
Within the Uniform Dimension Growth Conjecture \ref{conj_UDGC}, it is also interesting to ask for explicit expressions for $C(V)$ (and $\ga(V)$, if relevant), in terms of the degree and dimension of $V$.  

Conjecture \ref{conj_DGC} was stated by Serre \cite[p. 178 after Thm. 4]{Ser97} and   Heath-Brown, with a statement   for (absolutely irreducible) hypersurfaces in  \cite[p. 227]{HB83} and both non-uniform and uniform versions for hypersurfaces in \cite[Conj. 1, Conj. 2]{HB02}; see also a uniform version for varieties of any codimension in \cite[Conj. 2]{BHBS06}.
By (\ref{projI_DGC_relation}), Conjectures \ref{conj_DGC} and \ref{conj_UDGC} would imply  $N_{\Pbb^{n-1}}(M,B)\leq C(V)B^{\dimP V}s(B)$ for any projective thin set $M \subset V(\Q)$ of type I  of degree $d\geq 2$.

We turn our attention to surveying  works that recently resolved the Dimension Growth Conjecture \ref{conj_DGC}, as well as the  Uniform Dimension Growth Conjecture \ref{conj_UDGC}, in all but one case.

To prove Conjectures \ref{conj_DGC} and \ref{conj_UDGC} we may   restrict attention to varieties defined as the zero locus of forms that are all defined over $\Q$, as remarked in \cite[p.546]{BHBS06}:
\begin{lem}\label{lemma_field_of_dfn}
  Let $V \subset \Pbb^N$ be a projective variety, or let $V \subset \Abb^N$ be an affine variety, defined over a finite extension $k/\Q$. If the ideal of $V$ is not invariant under
the action of the absolute Galois group $\Gal(\overline{\Q}/\Q)$, then there exists a codimension 1 subvariety $Y \subset V$ with $\deg Y \ll_{k,N,\deg V} 1$ such that $V(\Q) \subseteq Y(\Q)$.
\end{lem}
Consequently, by working within each irreducible component $V_j$ of $V$ to which the lemma applies, $V_j(\Q) \subseteq Y_j(\Q)$ for  some codimension 1  subvariety $Y_j \subset V_j$, with $\deg Y_j \ll_{k,N,\deg V_j} 1$. Then the trivial bound $N_{\proj}(V_j,B)\ll_{\deg Y_j, \dimP Y_j,N} B^{\dimP Y_j+1}  \ll B^{\dimP V_j}$ of Lemma \ref{lemma_projI_trivial_bound} suffices to obtain Conjectures \ref{conj_DGC} and \ref{conj_UDGC} in this case.  (Here we note that the trivial bound in Lemma \ref{lemma_projI_trivial_bound} applies regardless of the field of definition of the variety.)
\begin{proof}
  Let $I(V)$ denote the ideal of $V$. Since $I(V)$ is not invariant under the action of $\Gal(\overline{\Q}/\Q)$, there exists an extension $k/\Q$ and a polynomial $f\in I(V)\subset k[\bX]$ such that $g (f)\not\in I(V)$ for some $g\in \Gal(\overline{\Q}/\Q)$.  Let $k = \Q(\alpha)$ and $[k:\Q]=d$; then we can decompose $f$ in the following way: 
    $$f(\bX) = \sum_{i=0}^{d-1} f_i(\bX)\alpha^i,$$
    where $f_i(\bX)\in \Q[\bX]$ for each $i,$ and $\deg f_i \ll_{k,N,\deg V} 1$.   For $\bx \in V(\Q)$, we must have that $f_i(\bx) = 0$ for each $i.$ Thus, to show that $V(\Q)$ lies in a proper subvariety, it suffices to show that $f_i(\bX)\not\in I(V)$ for some $i$. This follows from the fact that 
    $$(g (f))(\bX) = \sum_{i=0}^{d-1} f_i(\bX) \cdot (g (\alpha^i));$$
    if $f_i(\bX) \in I(V)$ for each $i$, then $g (f)\in I(V)$, a contradiction. Thus $f_i(\bX)\not\in I(V)$ for some $i$ but must vanish on all of $V(\Q)$, as desired.

\end{proof}

To prove Conjectures \ref{conj_DGC} and \ref{conj_UDGC}, it also suffices to consider only  the case of  geometrically integral varieties. Indeed:
\begin{lem}\label{lemma_typeI_geometrically_integral}
Let $V \subset \Pbb^N$ be a projective variety, or let $V \subset \Abb^N$ be an affine variety, defined over $\Q$, of degree $d$. If $V$ is integral but not geometrically integral, then there exists a codimension 1 subvariety $Y \subset V$ defined over $\Q$ of degree at most $d^2$ such that $V(\Q) \subseteq Y(\Q)$. 
\end{lem}
Consequently if $V \subsetneq \Pbb^N$ is integral but not geometrically integral then by Lemma \ref{lemma_projI_trivial_bound},  \[N_\proj(V,B) \ll_{\deg Y, \dimP Y,N} B^{\dimP Y+1} = B^{\dimP V},\] which suffices.    We state the lemma as in \cite[Prop. 2.1]{Ver24}, which cites the first paragraph of \cite[\S 4]{Wal15} for a proof of a similar result; see also \cite[Cor. 1]{HB02} and \cite[p. 1093]{Sal23} for analogous strategies. For completeness we record a proof.

\begin{proof}
    Assume $V$ is irreducible over $\mathbb{Q}$ but not over $\overline{\mathbb{Q}}$. Over $\overline{\mathbb{Q}}$ one can write
    \[
    V=\bigcup_{i=1}^\ell W_{i},
    \]
    where for every $i=1,\ldots,\ell$, the variety $W_{i}$ is irreducible over $\overline{\mathbb{Q}}$ and $W_{i}$ is not defined over $\Q$. 
    (Note that necessarily $\ell \geq 2$, since if $\ell=1$, then $W_{1}$ would be defined over the rationals, since $V$ is, yielding a contradiction.)

    Moreover notice that $\deg W_{i}\leq d$ for every $i=1,\ldots,\ell$. For each $i$ let $K_{i}/\mathbb{Q}$ be the minimal extension such that $W_{i}$ is defined over $K_{i}$ and let $\sigma_{i}\in \Gal (K_{i}/\Q)$ be such that $\sigma_{i}(W_{i})\neq W_{i}$. (Such a map exists since $W_{i}$ is not defined over $\Q$, and also $\deg \sigma_{i}(W_{i})\leq d$.) Then if $P\in\Q^{N}$ is a rational point  we have that
    \[
    P\in W_{i}(\Q)\Longleftrightarrow P\in \sigma_{i}(W_{i})(\Q).
    \]
    Hence,
    \[
    \begin{split}
    P\in V(\Q)&\Longrightarrow P\in W_{i}(\Q)\text{ for some }i\in\{1,\cdots,\ell\}\\&\Longrightarrow P\in (W_{i}\cap \sigma_i(W_{i}))(\Q)\text{ for some }i\in\{1,\cdots,\ell\}.
    \end{split}
    \]
    Thus
    \[
    V(\Q)\subset\bigcup_{i=1}^{\ell}(W_{i}\cap \sigma_i(W_{i}))(\Q).
    \]
    On the other hand, since $W_{i}$ is absolutely irreducible, it follows that $W_{i}\cap \sigma_i(W_{i})$ has codimension (at least) 1 in $W_i$. Finally, since both $\deg (W_{i}) \leq d$ and $\deg(\sigma_i(W_{i}))\leq d$, then $\deg (W_{i}\cap\sigma_i(W_{i}))\leq d^{2}$ by  B\'ezout's Theorem \cite[Cor. 2.4]{EisHar16}.
\end{proof}

 One of the most natural settings for studying the Dimension Growth Conjecture \ref{conj_DGC} is for $V \subset \mathbb{P}^{n-1}$   a nonsingular projective hypersurface of degree $d \geq 2$. This case was  resolved by combined works of Browning and Heath-Brown  for all $n \geq 3$ 
  \cite{HB94,HB02,Bro03a,BroHB06a,BroHB06b}.
In the general setting of a variety, one of the earliest results is due to Serre:  
\beq\label{Serre_projI}
N_\proj (V,B) \leq C(V) B^{\dimP V+1/2}(\log B)^{\ga(V)}
\eeq
for some $\ga(V)<1$. 
 (Serre deduced this in  \cite[p. 178 Thms. 3,4]{Ser97} from an upper bound for projective thin sets of type II, which is in turn deduced from studying affine thin sets of type II via (\ref{aff_proj_intro}); we defer further discussion to Remark \ref{remark_Serre_progress_DGC_via_projII}.)
To make further progress,
Browning, Heath-Brown, and Salberger showed that a bound for $N_{\proj}(V,B)$ can be obtained from a bound for a corresponding affine problem for hypersurfaces \cite[Thm. 1]{BHBS06}. That is, controlling projective thin sets of type I can be achieved by controlling certain affine thin sets of  type I, so we now turn our attention to that case.

 \subsection{Type I, affine}
Consider  $M \subset \Abb^{n}(\Z)$  a thin set of type I.   Let $V$ be a Zariski-closed subvariety $V \subsetneq \Abb^{n}$ such that $M \subset V(\Z)$. 
Since 
\[ N_{\Abb^{n}}(M,B) \leq \#\{ \bx \in V\cap \Abb^{n}(\Z):   \|\bx\| \leq B\},\]
improving on the trivial bound reduces to  a question about counting integral points of bounded height on an irreducible affine variety. Given any irreducible affine variety $V \subset \Abb^N$, define
\[ N_{\aff}(V,B) := \#\{ \bx \in V\cap \Abb^{N}(\Z):   \|\bx\| \leq B\}.\]
In the special case that $V$ is a hypersurface $H$ defined as the vanishing set of a polynomial $F(X_1,\ldots,X_N) \in \Z[X_1,\ldots,X_N]$, the trivial bound is also sometimes called the Schwartz-Zippel bound:
 \begin{lem}[Trivial bound, affine hypersurface] \label{lemma_Schwartz_Zippel}
 Let $A$ be a domain, such as $\Z$, or $\F_p$ for a prime $p$, or the ring of integers $\calO_k$ in a number field $k$. Let $F \in A[X_1,\ldots,X_n]$ be a nonzero   polynomial of degree $d \geq 1$, and $S \subset A$ a finite subset. Then 
\[ \#\{(x_1,\ldots,x_n) \in S^n : F(x_1,\ldots,x_n)=0\} \leq d |S|^{n-1}.\]
\end{lem}
The proof is by induction on dimension, and   may be found in many places, such as \cite[Thm. 1]{HB02} or \cite[Lemma 10.1]{BCLP23}. (While the proof given in \cite[Thm. 1]{HB02} is stated in the setting where $F$ is absolutely irreducible, and the proof given
in \cite[Lemma 10.1]{BCLP23} is stated in the case where $F$ is homogeneous, either proof applies in the present setting via trivial modifications.)

In the general where $V$ is a variety, a  trivial bound  is: 
\begin{lem}[Trivial bound, affine variety]\label{lemma_affineI_trivial_bound}
For any variety $V \subsetneq \Abb^N$,
\[
N_{\aff}(V,B) \ll_{\dimA V,\deg V,N} B^{\dimA V}.
\]
If moreover $V$ is of pure dimension, the dependence can be made explicit as 
\[N_{\aff}(V,B) \ll_{\dimA V,N} (\deg V)B^{\dimA V}.\]
\end{lem}
The exponent is best possible if $V$ is of degree 1.   In general if $V \subset \Abb^n$ is an affine variety, a nonuniform bound $N_{\aff}(V,B) \ll_{V} B^{\dimA V}$  can again be obtained by an argument that projects from $\Abb^N$ to $\Abb^{\dimA V}$, analogous to the remark following Lemma \ref{lemma_projI_trivial_bound}. For completeness, we now provide a proof of the uniform version recorded in the present lemma. This proof is similar to \cite[Thm. 1]{BroHB05}, and the result stated above can also be found for example in \cite[Lemma 4.1.1]{CCDN20}, \cite[Prop. 2.1]{BCSSV25x}. It is an interesting feature that the constant in the affine Schwartz-Zippel bound in Lemma \ref{lemma_Schwartz_Zippel} and Lemma \ref{lemma_affineI_trivial_bound} has linear dependence on $\deg V$. (If $V$ is a union of $d$ parallel linear varieties, this dependence on $d=\deg V$ cannot be improved.) In contrast, see Remark \ref{remark_Dehennin}.

\begin{proof}
For clarity, we suppose the given variety $V$ is defined over a field $k$. In this paper we only require the case $k=\Q$, or a number field $k$ in the application following Lemma \ref{lemma_field_of_dfn} via Lemma \ref{lemma_projI_trivial_bound}. Thus we assume here for simplicity that $k$ is a number field, although the proof can be adapted to more general settings, since it only relies on Lemma \ref{lemma_Schwartz_Zippel}. 

Decompose $V$ into irreducible components (over $k$) as $V = \cup_{j \in J} W_j$, in which  there are at most $\deg V$  components. Consequently, to prove the main claim of the lemma, it is sufficient to prove that for any irreducible variety $W \subset \Abb^{N}$ defined over $k$,
\beq\label{affine_SZ_IH}
N_{\aff}(W,B) \ll_{\dimA W,N}(\deg W) B^{\dimA W}.
\eeq 
Applying this to each irreducible component in $V$ leads to 
\[N_{\aff}(V,B) \ll_{\dimA V,N}   (\sum_{j \in J} \deg W_j)B^{\dimA V}.\]
Recall that the degree of $V$ is the sum of the degrees of the irreducible components of $V$ of highest dimension. We may always conclude that $\sum_{j \in J}  (\deg W_j) \leq (\deg V)^2$, and this suffices for the main claim of the lemma. In the case that $V$ is of pure dimension, so that  $\dimA W_j=\dimA V$ for each $j$, then $\sum_{j \in J}  (\deg W_j)=\deg V$, and this leads to the second conclusion of the lemma.

We  proceed to prove (\ref{affine_SZ_IH}) for irreducible $W$,   by induction on the dimension of the ambient space $\Abb^N$.    If $N=1$ and $W \subsetneq \Abb^1$ then  $\#\{ \bx \in W\cap \Abb^{1}(\Z) \}  \leq \deg W$, so that the conclusion of the lemma holds (with $\dimA W=0$). We assume the inductive hypothesis that 
(\ref{affine_SZ_IH}) holds  
for any irreducible variety $W\subsetneq \Abb^{N-1}$, and we prove it for any irreducible variety $V \subsetneq \Abb^{N}$. Consider such a $V \subsetneq \Abb^{N}$. If $\dimA V=N-1$ so that $V$ is a hypersurface, the result of the lemma already holds by the Schwartz-Zippel trivial bound in Lemma \ref{lemma_Schwartz_Zippel}. Thus from now on we may assume that $\dimA V \leq N-2$.

For each $x_{1}\in\mathbb{\Z}$ consider the hyperplane $H(x_1)$ cut out by $X_{1}-x_{1}=0$. Such hyperplanes are disjoint, i.e. $H(x_1)\cap H(x_1')=\emptyset $ if $x_{1}\neq x_{1}'$. Now if a hyperplane $H$ does not contain $V$, then $V\cap H$ is a proper closed subset of $V$ and since $V$ is irreducible, $\dimA (V\cap H)\leq \dimA V-1$. Therefore $\dimA (V\cap H(x_1))=\dimA V$ is possible only if $H(x_1)$ contains $V$ but this can occur for at most one value of $x_{1}$ since the hyperplanes $\{H(x_1)\}_{x_{1}}$ are pairwise disjoint. 

Suppose now $V$ lies in (at most) one hyperplane, say $H(x_1^*)$ for a fixed $x_1^* \in \Z$, and  for all $x_1 \in \Z \setminus \{x_1^*\}$,    $\dimA (V\cap H(x_1))\leq \dimA V -1$. Hence upon writing $\bx=(x_1,\bfy)$ for $\bfy=(x_2,\ldots,x_{N})$,  
\begin{align*}
N_\aff (V,B) & =\sum_{\substack{\|\bfx\|\leq B\\\bfx\in V}}1 
= \sum_{\substack{|x_{1}|\leq  B\\ x_1 \neq x_1^*}}\sum_{\substack{\|\bfy\|\leq B\\\bfy\in V\cap H(x_1)}}1 +  \sum_{\substack{\|\bfy\|\leq B\\\bfy\in V\cap H(x_1^*)}}1 \\
& \ll_{ \dimA V, N} (\deg V)B \cdot B^{\dimA V-1} + (\deg V) B^{\dimA V}.
\end{align*}
Here in the first term, in the sum over $\bfy$ we have used the fact that $H(x_1)\equiv \mathbb{A}^{N-1}$ and  $\dimA(V\cap H(x_1)) \leq \dimA V-1 \leq N-2$, so we may apply the inductive hypothesis to the sum over $\bfy$.  
In the second term, in the sum over $\bfy$ we again   used   $H(x_1^*)\equiv \mathbb{A}^{N-1}$, as well as the fact that we have already reduced to the case $\dimA V \leq N-2$. Thus even though $\dimA (V \cap H(x_1^*)) = \dimA V$, 
this sum over $\bfy$ is still subject to the inductive hypothesis, since 
  $\dimA(V\cap H(x_1^*))\leq \dimA V  \leq N-2 <N-1$.
Thus the proof  is complete.
  \end{proof}   

 \begin{rem}
 
 Another method to achieve a bound with the exponent in  Lemma \ref{lemma_affineI_trivial_bound}   (and correspondingly in Lemma \ref{lemma_projI_trivial_bound}) is to apply the Lang-Weil bound (\emph{a priori} stronger than the Schwartz-Zippel upper bound, since it implies a lower bound).  For example, in the affine setting, consider    a variety $V\subset \Abb^N$ defined over $\Q$. Suppose for a given $B$ there is a prime $p \in [B/2,B]$ such that $V$ is integral over $\overline{\F}_p$ (and of the same dimension). Then by the Lang-Weil bound (in the affine formulation of Lemma \ref{lemma_Lang_Weil_cor} below), $\#V(\F_p)=p^{\dimA V} + O_{\dimA V,\deg V}(p^{\dimA V-1/2})$. Since at most $B/p+O(1)=O(1)$ points $\bx \in [-B,B]^n \cap \Z^n$ project down to each point in $\F_p^n$,   we may conclude that $N_\aff(V,B) \ll_{\dimA V,\deg V} p^{\dimA V}  \ll B^{\dimA V}.$ As stated, this bound is independent of the coefficients of the defining polynomials of $V$, but it assumed the existence of a prime $p$ of good reduction of size comparable to $B$. Suppose for simplicity that $V$ is a hypersurface defined as the vanishing set of a  polynomial $F$ with integral coefficients, and suppose that $F$ is irreducible over $\overline{\Q}$. Then by a Bertini-Noether lemma (such as \cite[Lemma 2.9]{BPW25x}), $F$ is irreducible over $\overline{\F}_p$ for all but $\ll_{N,\deg F} \log \|F\|/\log \log \|F\|$ primes. In this case, the argument above would only apply as long as $B\gg \log \|F\|$, so that $\pi(B)\gg \log \|F\|/\log \log \|F\|$, which would imply the existence of a suitable prime $p \in [B/2,B]$. 
    
\end{rem}

As a result of Lemma \ref{lemma_affineI_trivial_bound}, for a thin set $M \subset V(\Z) \subsetneq \Abb^n(\Z)$, the trivial bound is
\beq\label{affineI_trivial_thin_set}
N_{\Abb^n}(M,B) \ll_{\dimA V,\deg V,n} B^{\dimA V} \ll B^{n-1}.
\eeq 
Pila  has obtained a celebrated uniform bound: for any irreducible affine variety $V \subset \Abb^N$ defined over $\R$ of degree $d$,
\beq\label{Pila_affineI}
N_{\aff}(V,B) \ll_{\dimA V,d,N,\ep} B^{\dimA V-1+1/d+\ep},
\eeq
for any $\ep>0$. See \cite[Thm. A]{Pil95}; this builds  on the Bombieri-Pila argument for plane curves in \cite{BomPil89}. In fact, Pila proves a  stronger result, with $B^\ep$ replaced by $\exp(12\sqrt{d\log B \log \log B})$.  Explicit refinements are now known. Cluckers, D\`ebes, Hendel, Nguyen and Vermeulen \cite[Prop. 2.6]{CDHNV23x} quantified the dependence on $d$   in (\ref{Pila_affineI}), proving 
\[ N_{\aff}(V,B) \ll_N d^{e(N,\dimA V)} B^{\dimA V-1+1/d}\log B\]
 for some integer $e(N,\dimA V)$. In the case when $V$ is an (irreducible affine) hypersurface $H$, Binyamini, Cluckers and Kato \cite[Thm. 3]{BCK25} showed at most quadratic dependence on degree (at the cost of more log factors), proving 
\[ N_{\aff}(H,B) \ll_N d^2 B^{\dimA H-1+1/d}(\log B)^{\ga(N)}\]
for some $\ga(N) \geq 0$.
By \cite[Prop. 4.3.1]{CCDN20}   a result on $N_\aff(H,B)$ for hypersurfaces implies a result for $N_{\aff}(V,B)$ for varieties, but with an increased power of $d$. Precisely, for each geometrically integral $V\subset\Abb^{N}$ of dimension $m$ and degree $d$ there is a geometrically integral hypersurface $H\subset \Abb^{m+1}$ (hence of dimension $m$) of degree $d$ and birational to $V$ such that
\[
N_\aff(V,B) \leq d N_\aff(H,c(n)d^{e(n)} B),
\] 
for some $c(n),e(n)$.

The exponent in (\ref{Pila_affineI}) (and the aforementioned explicit refinements) improves on Lemma \ref{lemma_affineI_trivial_bound} as long as $d\geq 2$, and indeed the exponent is essentially sharp for each $d$, as the following example shows.
\begin{example}\label{example_affineI_lower_bound}
 Let $V\subset \Abb^n$ be the hypersurface defined by $F=0$ for $F(X_1,\ldots,X_n) = X_1^d - (X_2 + \cdots + X_n)$.  
 Then $\dimA V=n-1$ and 
$N_\aff(V,B)  \gg_{n,d} B^{\dimA V-1 +1/d}$, since 
\[
 N_\aff(V,B) = \sum_{\substack{|z|\leq (n-1)B \\ \text{$z$ a $d$-th power}}} \sum_{\substack{x_2+\cdots+x_n=z \\ (x_2,\ldots,x_n) \in [-B,B]^{n-1}}} 1 \gg_{n,d} B^{1/d} \cdot B^{n-2} .
\]
\end{example}

As a consequence of this example, one could say that  (\ref{Pila_affineI})  concludes the   study of affine thin sets of type I.
However, before we conclude our investigation prematurely, let us recall that we turned to affine thin sets of type I with the promise of understanding projective thin sets of type I. We record   the bridge between these settings  as in \cite[Thm. 1]{BHBS06}, which is proved by a birational projection argument (related to the fact that any variety $X$ of dimension $r$ is birational to a hypersurface in $\Pbb^{r+1}$, see \cite[Ch. I Prop. 4.9]{Har77}).  
\begin{prop}\label{prop_BHBS06_affineI_projI}
Consider a hypersurface $H \subset \Abb^n$ with $n \geq 3$  defined by the vanishing of a homogeneous form $F(X_1,\ldots,X_n) \in \Z[X_1,\ldots,X_n]$ of degree $d \geq 2$ that is absolutely irreducible. Let $\ep>0$. Suppose that for all such affine hypersurfaces $H$, 
\beq\label{affineI_hypersurface_upper}
N_\aff(H,B)\ll_{n,d,\ep} B^{\theta_{d,n}+\ep}.
\eeq
Then for any geometrically integral projective variety $V \subset \Pbb^N$ of degree $d \geq 2$ and $\dimP V=n-2$,
\[ N_\proj(V,B) \ll_{N,d,\ep} B^{\theta_{d,n}+\ep}.\]
\end{prop} 
Consequently, to make progress on the Dimension Growth Conjectures \ref{conj_DGC} and \ref{conj_UDGC}, it suffices to sharpen Pila's bound (\ref{Pila_affineI}) only in certain cases, which exclude in particular the behavior exhibited by Example \ref{example_affineI_lower_bound}. Here is the relevant   conjecture, modeled on \cite[Conj. 1 and Conj. 3]{BHBS06}. Throughout, the leading form of a polynomial is the homogeneous part of highest degree.
\begin{conj}[Uniform Dimension Growth, affine hypersurface]\label{conj_affine_hypersurface}
 For any hypersurface $H \subset \Abb^n$ defined as the vanishing set of a  polynomial $F \in \Z[X_1,\ldots,X_n]$ of degree $d \geq 2$ with absolutely irreducible leading form, 
   \[ N_\aff(H,B) \ll_{n,d}  B^{\dimA H-1}s(B),\]
   for a ``small factor'' $s(B)$, which could take the form   $s(B)\ll_\ep B^\ep$ (for any $\ep>0$); or $s(B) \leq (\log B)^{\ga(n,d)}$ for some $\ga(n,d) > 0$; or $s(B) \leq 1$.
\end{conj}

The main exponent is best possible:
\begin{example}
    For every $n \geq 3$ and $d \geq 2$ there exists a hypersurface $H$ with $N_{\aff}(H,B) \gg_n B^{n-2}$. For example, consider $H \subset \Abb^n$ defined by the vanishing of $X_1F_1(\bX) - X_2 F_2(\bX)$ for forms $F_1,F_2 \in \Z[X_1,\ldots,X_n]$ of degree $d-1$; then the form vanishes for all $(0,0,x_3,\ldots,x_n)$ with $|x_j| \leq B$ for $j=3,\ldots,n$; see \cite[Conj. 1]{BHBS06}.
    \end{example}

In order to apply Proposition \ref{prop_BHBS06_affineI_projI}, we need only consider the case of Conjecture \ref{conj_affine_hypersurface} when $F$ is homogeneous, and we currently restrict our attention to this case, returning to the general setting in \S \ref{sec_revisit_affineI}.
When $F$ is homogeneous, Conjecture \ref{conj_affine_hypersurface} was established by Heath-Brown for  $n \leq 4$  \cite[Thms. 3, 9]{HB02} and in the cases $d=2$,   $n \geq 3$ \cite[Thm. 2]{HB02}, with $s(B) \ll_\ep B^\ep$ for any $\ep>0$. With this same notion of small factor,
Browning, Heath-Brown, and Salberger proved the conjecture in the homogeneous case for $n \geq 3$ and $d \geq 6$ \cite[Cor. 2]{BHBS06} (and certain weaker results for $3 \leq d \leq 5$). 
Then Salberger \cite[Thm. 0.4]{Sal23} proved Conjecture \ref{conj_affine_hypersurface} for all $d \geq 4$. Cumulatively:
\begin{thm}[Browning, Heath-Brown, Salberger]\label{thm_affine_hypersurface}
    The Uniform Affine Dimension Growth  
  Conjecture \ref{conj_affine_hypersurface} holds for all homogeneous $F$ with  degree $d \neq 3$, with small factor $s(B) \ll_\ep B^\ep$ for all $\ep>0$.
\end{thm} 
 It remains an open problem to verify the uniform statement in
  Conjecture \ref{conj_affine_hypersurface} for $F$ homogeneous of degree $d=3$, for example with small factor $s(B) \ll_\ep B^\ep$. For degree $d=3$, define 
  \beq\label{del3_dfn}
  \del_3 := 2/\sqrt{3}-1, \qquad \del_3':= 1/7< \del_3.
 \eeq
    Salberger  proved  that under the hypotheses of Conjecture \ref{conj_affine_hypersurface} for $d=3$, there is a uniform bound
  \beq\label{Salberger_affine_3}
  N_{\aff}(H,B) \ll_{n,\ep}B^{\dimA H-1+\del+\ep} ,
  \eeq
  with $\del=\del_3$ obtained in \cite[Thm. 0.4]{Sal23}, and the smaller value $\del = \del_3'$ obtained in \cite[Thm. 16.3]{Sal15}; see also \cite[Thm. 16.4]{Sal15} for  geometrically integral projective subvarieties. (For clarity, note that \cite{Sal15} was developed later than, but published before, \cite{Sal23}.)

 \subsection{Type I, revisiting the projective setting}
We now return to the setting of projective sets of type I. By Theorem \ref{thm_affine_hypersurface} combined with Proposition \ref{prop_BHBS06_affineI_projI}:
\begin{thm}[Browning, Heath-Brown, Salberger]\label{thm_UDGC_degneq3}
    The  Uniform Dimension Growth Conjecture \ref{conj_UDGC} holds for all cases with degree $d \neq 3$, with $C(V)$ depending only on the dimension and degree of $V$ and with small factor $s(B) \ll_{\ep}B^\ep$ for any $\ep>0$. 
\end{thm}
As a result of (\ref{Salberger_affine_3}), Salberger  has obtained a uniform but weaker statement toward Conjecture \ref{conj_UDGC} that for any geometrically integral variety $V \subset \Pbb^N$ of degree 3,
\[ N_\proj(V,B) \ll_\ep B^{\dimP V+\del+\ep},\]
for $\del=\del_3$ \cite[Thm. 0.3]{Sal23} or the smaller $\del = \del_3'$ \cite[Thm. 16.3, Cor 16.4]{Sal15}.
It remains an open problem to improve the exponent to achieve Conjecture \ref{conj_UDGC} for degree 3.

Salberger has also established the non-uniform  Dimension Growth Conjecture for $\deg d=3$ in \cite[Thm. 0.1 and \S 8]{Sal23} by ``ad hoc'' methods (not passing through Proposition \ref{prop_BHBS06_affineI_projI}):
 \begin{thm}[Salberger]\label{thm_DGC_deg3}
      The   Dimension Growth Conjecture \ref{conj_DGC} holds for all cases,   with   small factor $s(B) \ll_{\ep}B^\ep$ for any $\ep>0$. 
 \end{thm}

Recent studies on Uniform Dimension Growth have further produced explicit expressions for $C(V)=C(\deg V, \dimP V)$, and the small factor $s(B)$, beginning with influential work by Walsh \cite{Wal15}, which replaced $s(B)\ll_\ep B^\ep$ by $s(B)\ll 1$ in a key result counting points on projective curves, underlying the approach to Theorem \ref{thm_UDGC_degneq3}. Explicit refinements were continued by Castryck, Cluckers, Dittmann, Nguyen \cite{CCDN20}, who established polynomial dependence on degree in $C(V)$. 

Let us briefly motivate, at least informally, why it can be useful to make the dependence of the constant $C(V)$ on degree explicit, and to minimize it as much as possible. 
For example, given  a geometrically integral projective hypersurface $X \subset \Pbb^{r+1}$ of degree $d$ defined over $\Q$, a version of the determinant method may produce for each $B \geq 1$ a geometrically integral hypersurface $Y_B$ of (potentially very large) degree $D_{Y_B}$,   with the properties that $X \not\subset Y_B$, and   all rational points on $X$ of height at most $B$ lie on $Y_B$. (For example, a result of Salberger produces such a $Y_B$ with degree 
\beq\label{degree_Y_Salberger} 
D_{Y_B}\ll_{d,r,\ep}B^{(r+1)/(rd^{1/r})+\ep};
\eeq
see \cite[Thm. 0.12]{Sal23}. In particular, this yields $B^{3/(2\sqrt{d})+\ep}$ for $r=2$.) 
Then,
$$\{\bx \in X(\Q):H(\bx)\leq B\} \subset Y(\Q) \implies \{\bx\in X(\Q):H(\bx)\leq B\} \subset (X\cap Y)(\Q).$$

Suppose from now on that $X \subset \Pbb^3$ (that is, $r=2$).
Now since $X\not\subset Y_B$, $X\cap Y_B =: Z$ is a curve of degree $D_{Y_B}$. Suppose now that for some function $C(\cdot)$, it is true that for any geometrically integral curve $Z$, 
$$  \#\{\bx\in Z(\Q):H(\bx)\leq B\} \ll_\ep C(\deg Z)  B^{ 2/\deg Z +\ep}.$$
(For example, a result of Heath-Brown produces this upper bound with unspecified dependence $C(\deg Z)$, for any geometrically integral curve in $\Pbb^3$ \cite[Thm. 5]{HB02}.) 
Combining these two results to answer the original question, we would arrive at a bound like 
$$  \#\{\bx\in X(\Q):H(\bx)\leq B\} \ll_\ep C(D_{Y_B})  B^{ 2/D_{Y_B} +\ep}.$$
So long as $D_{Y_B}$ is a function that grows to infinity with $B$, this upper bound becomes  $\ll_\ep B^\ep C(D_{Y_B})$ as $B \maps \infty$. Thus on the one hand we are motivated to show $D_{Y_B} \maps \infty$ with $B$, and on the other hand we need to show that the function $C(\cdot)$ is as slow-growing as possible. (For example in $\Pbb^3$, if $C(x) \ll x^2$ were quadratic and $D_{Y_B}\approx B^{3/(2\sqrt{d})+\ep}$, this line of reasoning  would lead to an upper bound  
\[ \#\{\bx\in X(\Q):H(\bx)\leq B\} \ll_\ep (B^{3/2\sqrt{d}+\ep})^2B^{ 2/D_{Y_B} +\ep} \ll_\ep B^{3/\sqrt{d}+\ep}.\]
Any improvement on the degree in (\ref{degree_Y_Salberger}) would lead to a corresponding improvement here. See for example \cite[Thm. 3.1, Cor 3.3]{BCSSV25x}, which (roughly speaking) allows $D_{Y_B} \ll_d B^{1/\sqrt{d}}(\log B)$.
Motivated by strategies like this, the goal of proving at most quadratic dependence on degree was raised by Salberger in remarks following \cite[Thm. 0.12]{Sal23}.

There have been many results recently that make the dependence on degree more explicit. We record from \cite[Thm. 1]{CCDN20}, for all geometrically integral varieties $V\subset \Pbb^{N}$ of degree $d \geq 5$,
\beq\label{no_small_factor}
N_\proj (V,B) \ll_N d^{e(N)}B^{\dimP V}
\eeq
for some integer $e(N)$.
  Cluckers, D\`ebes, Hendel, Nguyen, Vermeulen \cite[Thm. 1.2]{CDHNV23x} proved $e(N)=7$ suffices if $V$ is a hypersurface. The best-known bounds show at most quadratic dependence on $d$,  due to Binyamini, Cluckers, and Novikov \cite{BCN24} (over $\Q$) and Binyamini, Cluckers and Kato \cite{BCK25} (including results over any global field). Precisely, as in \cite[Thm. 4]{BCK25},
 for any irreducible hypersurface $H$ of degree $d \geq 4$ in $\Pbb^N$, 
 \beq\label{BCK_statement}
 N_\proj(H,B) \ll_N d^2 B^{\dimP H}(\log B)^{\ga(N)}
 \eeq
 for some  $\ga(N)$, while for 
 degree $d=3$,
 \[ N_\proj(H,B) \ll_N  B^{\dimP H + \del_3}(\log B)^{\ga(N)},\]
 for $\del_3$ as in (\ref{del3_dfn}).
It remains an open problem to improve the exponent to $\dimP H$ for degree $3$. These results for hypersurfaces imply results for varieties, but with an increased power of $d$: by  \cite[Prop. 4.3.2]{CCDN20}, for each geometrically integral $V\subset\Pbb^{N}$ with $\dimP V=m$ and degree $d$ there is a geometrically integral hypersurface $H\subset \Pbb^{m+1}$ (hence of dimension $m$) of degree $d$ and birational to $V$ such that
\beq\label{Nproj_pass_from_hypersurface_to_V}
N_\proj(V,B) \leq d N_\proj(H,c(n)d^{e(n)} B)
\eeq
for some $c(n),e(n)$.

As noted in \cite[\S 1.4]{CCDN20}, by Example \ref{example_quadric}, (\ref{no_small_factor}) cannot be true in general for degree $d=2$. It remains open whether a bound with small factor $s(B) \ll 1$ could be true for degrees $d=3,4$.

Lower bounds have also been studied. By \cite[\S 6]{CCDN20}, for every degree $d \geq 1$ there is a geometrically integral degree $d$ hypersurface $H \subset \Pbb^N$ such that $N_\proj(H,B) \gg_N d B^{\dimP H}$, so that in general one cannot expect to achieve smaller than linear growth in $d$. This has been improved by Cluckers and Glazer \cite[Thm. 2]{CluGla25} to the statement that quadratic growth in $d$ is ``eventually tight'' (as $N \maps \infty$). More precisely,   for any $N \geq 2$ and $B_0 \geq 1$, there exists $c(N)>0$, $B \geq B_0$ and $d > B_0$, and an integral hypersurface $H \subset \Pbb^N$ of degree $d$ such that 
\[ N_\proj(H,B) \geq c(N) d^{2-2/N} B^{\dimP H}.\]
It remains an open question to understand the small discrepancy in dependence on degree, compared to the upper bound (\ref{BCK_statement}); here very recent work of Dehennin is making progress \cite{Deh26x}.

 \subsection{Type I, revisiting the affine setting}\label{sec_revisit_affineI}
 We now revisit the affine setting of Dimension Growth results. We have observed in Example \ref{example_affineI_lower_bound} that the exponent in Pila's bound (\ref{Pila_affineI}) cannot be improved if $V \subset \Abb^n$ is allowed to be an arbitrary irreducible variety, although Theorem \ref{thm_affine_hypersurface} improves it when $V$ is a hypersurface defined as the vanishing set of an absolutely irreducible form $F$ (of degree $d \neq 3$). It is natural to ask for the largest class of ``allowable'' affine varieties $V \subset \Abb^n$ of degree $d \geq 2$ for which  can (\ref{Pila_affineI}) be improved to 
 \beq\label{Pila_improvement}
 N_\aff(V,B) \leq C(n,d) B^{\dimA V-1}s(B),
 \eeq
  for a suitable small factor $s(B)$, in the terminology of Conjecture \ref{conj_affine_hypersurface}. We turn now to progress on this question.

\subsubsection{Hypersurfaces}
  We first consider the setting where $V$ is a hypersurface $H$. 
 Salberger achieved (\ref{Pila_improvement}) (with $s(B) \ll_\ep B^\ep$ for every $\ep>0$) for any  affine hypersurface $H \subset \Abb^n$ defined as the vanishing set of a polynomial $f \in \Z[X_1,\ldots,X_n]$ of degree $d \geq 4$ and $n \geq 3$, with leading form (homogeneous part of highest degree) that is   absolutely irreducible, hence verifying Conjecture \ref{conj_affine_hypersurface} in these cases \cite[Thm. 0.4]{Sal23}. (We believe the hypothesis of that theorem intended to require absolute irreducibility, as for example in \cite[Thm. 7.4]{Sal23}.)   To interpret this condition geometrically, if $f$ is a polynomial of degree $d$ in $X_1,\ldots,X_n$, let $F(T,X_1,\ldots,X_n)$ denote its homogenization. Then to say the leading  form $f_d$ of $f$ is absolutely irreducible is to say that the intersection of the homogenization $F(T,X_1,\ldots,X_n)=0$ with the plane at infinity $\{T=0\}$ is geometrically irreducible, since $V(F)\cap V(T)=V(f_d)$.

 It remains an open problem to prove an analogous result when $f$ has degree $d=3$, with leading form that is absolutely irreducible.
Instead, Salberger has proved 
\[  N_{\aff}(H,B) \ll_{n,\ep}B^{\dimA H-1+\del+\ep} ,
 \]
  with $\del=\del_3$ as in (\ref{del3_dfn}) obtained in \cite[Thm. 0.4]{Sal23}, and the smaller value $\del = \del_3'$ obtained in \cite[Thm. 16.3]{Sal15}. 
Dependence on degree  in (\ref{Pila_improvement}) for $V$ a hypersurface has been made explicit by \cite{BCN24} and \cite{BCK25}: as in \cite[Thm. 5]{BCK25}, one may take $C(n,d) \ll_n d^2$ and $s(B) \leq (\log B)^{\ga(n)}$ for some constant $\ga(n)$, for degree $d \geq 4$; while if $d =3$, they obtain 
\beq\label{Naff_deg3}
N_\aff(H,B) \ll_n B^{\dimA H-1 + \del_3}(\log B)^{\ga(n)}.\eeq

Next, consider the case where $V$ is a hypersurface $H$, but we aim to relax the condition that the leading form of the defining polynomial of $H$ is absolutely irreducible. The work \cite{CDHNV23x} replaces this   with a condition that the hypersurface is \emph{not cylindrical over a curve (NCC)}. 
Precisely, as in \cite[Dfn. 1.3]{CDHNV23x},   a hypersurface $H \subset \Abb^n$ is said to be cylindrical over a curve if there exists a $\Q$-linear map $L: \Abb^n(\Q) \maps \Abb^2(\Q)$ and a curve $C$ in $\Abb^2(\Q)$ such that $H=L^{-1}(C)$. Otherwise, $H$ is not cylindrical over a curve. A polynomial $f \in \Q[X_1,\ldots,X_n]$  is said to be NCC precisely when the hypersurface $H\subset \Abb^n$ defined as the vanishing set $V(f)$ is NCC.

\begin{rem}[On being cylindrical over a curve]\label{remark_NCC} 
As remarked in \cite{CDHNV23x}, the property that a polynomial $f$ is cylindrical over a curve is equivalent to saying there exists a polynomial $g \in \Q[Y_1,Y_2]$ and linear forms $L_1(\bX),L_2(\bX) \in \Q[X_1,\ldots,X_n]$ such that $f(\bX) = g(L_1(\bX),L_2(\bX))$. If this holds, then by a $\GL_n(\Q)$ change of variables, $f(\bX)$ can be made to depend on (at most) two variables.
Thus if $f \in \Z[X_1,\ldots, X_n]$ with $n \geq 3$ is NCC, then $f$ depends on at least 3 variables 
 after any affine linear transformation; we will see  in \S \ref{sec_revisit_affineII} that related ``nondegeneracy'' conditions are also useful when proving  upper bounds for affine thin sets of type II as in \cite{BPW25x}.

 If the leading form of $f$ is absolutely irreducible, then $f$ must be NCC. For suppose $f$ is cylindrical over the curve $C = \{g(x,y)=0\}$, so that $f = g(L_1(\bx),L_2(\bx))$. Let $g_0$ denote the leading form of $g$, which cannot be absolutely irreducible since it is a binary form (see e.g. \cite[Remark 5.3]{BonPie24}). Thus the leading form of $f$, namely $g_0(L_1(\bx),L_2(\bx))$ is never absolutely irreducible when $f$ is cylindrical over a curve. On the other hand, there are many polynomials whose leading forms are not absolutely irreducible that are also not cylindrical over curves.
 \end{rem}

Now consider the class of hypersurfaces $H$ that are NCC. Additionally, assume that $H \subset \Abb^n$ is defined as the vanishing set of a polynomial $f \in \Z[X_1,\ldots,X_n]$ of degree $d \geq 4$ such that $f$ is irreducible over $\Q$ and has leading form $f_d$ with no linear factors over $\Q$ (in which case $f_d$ is said to be 1-irreducible).   For this class, \cite[Thm. 1.5]{CDHNV23x} shows
(\ref{Pila_improvement}) with $C(n,d) \ll_n d^7$ and small factor $s(B) \leq (\log B)^{\ga(n)}$. If the hypothesis is strengthened to assume the leading form $f_d$ of $f$ has no factors of degree $\leq 2$ over $\Q$ (in which case $f_d$ is said to be 2-irreducible), this can be strengthened to $s(B) \leq 1$ for $d \geq 5$. 
Under hypotheses such as these, it remains an open problem to achieve (\ref{Pila_improvement}) for degree $d=3$. 
So far, \cite[Thm. 1.5]{CDHNV23x} has achieved the bound (\ref{Naff_deg3}) if $H$ is a hypersurface that is NCC and $f$ is of degree $3$ with 1-irreducible leading form $f_3$.

Regarding lower bounds, by \cite[\S 6]{CCDN20}, for every $d\geq 1$ there exists a hypersurface $H \subset \Abb^n$ defined as the vanishing set of an absolutely irreducible degree $d$ polynomial $f \in \Q[X_1,\ldots,X_n]$ such that $N_\aff(H,B) \gg_n d^2 B^{n-2}$, so that in general one cannot expect to prove smaller than quadratic growth in $d$. 

In contrast, if $f \in \Z[X_1,\ldots,X_n]$ is cylindrical over a curve and we consider the hypersurface $H \subset \Abb^n$ defined as the vanishing locus of $f$,  then  
\beq\label{hypersurface_compare_curve} 
N_\aff(H,B) \ll_n B^{n-2}N_\aff(C,B)
\eeq
for a curve $C \subset \Abb^2(\Q)$ \cite[Remark 1.6]{CDHNV23x}. The resulting curve $C$ could contain many points; for example $C: x_1=x_2^d$ contains $\gg B^{1/d}$ points  $(x_1,x_2) \in [-B,B]^2 \cap \Z^2$, in which case one could not expect to improve on the exponent in Pila's bound (\ref{Pila_affineI}). This shows that the NCC condition cannot in general be removed   from the above results. But   a conjectural approach in \cite[Thm. 1.9]{CDHNV23x}  suggests that $N_\aff(H,B) \ll_{n,d,\ep} B^{\dimA H-1+\ep}$ may be true for all affine hypersurfaces $H \subset \Abb^n$ defined as the vanishing set of a polynomial $f \in \Z[X_1,\ldots,X_n]$ of degree $d \geq 2$ and $n\geq 2$ such that the leading form $f_d$ of $f$ has no linear factors over $\Q$. 

\subsubsection{Varieties}
We now return to the question of proving (\ref{Pila_improvement}), for more general varieties than hypersurfaces. Here Vermeulen has generalized the notion of NCC. Let $V \subset \Abb^n$ be an affine variety of dimension $m$. Then $V$ is said to be cylindrical over a curve if there exists a $\Q$-linear map $L: \Abb^n \maps \Abb^{n-m+1}$ such that $L(V)$ is a curve; otherwise, $V$ is said to be NCC. Under the hypothesis that $V \subset \Abb^n$ is NCC of degree $d \geq 4$, \cite[Thm. 1.2]{Ver24} has proved (\ref{Pila_improvement}) with small factor $s(B) \ll_\ep B^\ep$ for any $\ep>0$ (and unspecified $C(n,d)$). It remains an open question to achieve this for degree $d=3$. So far, in this setting, (\ref{Naff_deg3}) is obtained with a $B^\ep$ factor, if $d=3$.

On the other hand, if $V$ is cylindrical over a curve, and so there is a $\Q$-linear map $L: \Abb^n \maps \Abb^{n-m+1}$ such that $L(V)$ is a curve, then
 \[ N_\aff(V,B) \leq B^{\dimA V-1}N_\aff(L(V),B)\] 
 so that consideration reduces to counting points on the curve $L(V)$, leading to similar considerations as in (\ref{hypersurface_compare_curve}). 
 
 It remains an open question to characterize the largest ``allowable'' class of affine varieties for which the main exponent in (\ref{Pila_improvement}) is achievable.
This concludes our discussion of Dimension Growth results, relevant to thin sets of type I in both the projective and affine settings.

\section{Survey of   type II thin sets}\label{sec_lit_typeII}
 We now shift focus to thin sets of type II. Our path will follow a similar pattern to that above: studying the main conjecture in the projective setting will require results from the affine setting;   we must restrict attention to special cases in the affine setting in order to obtain suitably strong results to  attain the central conjecture in the projective setting; and finally we can ask for the largest ``allowable'' class in the affine setting for which strong results (analogous to the main conjecture in the projective setting) can be obtained. 

 \subsection{Type II, projective}
 By definition, if $M$
  is of type II
 then there is an irreducible projective  algebraic variety $Z$ over $\Q$ with $\dimP  Z = \dimP  (\mathbb{P}^{n-1})$, and a  dominant morphism $\pi: Z \rightarrow \mathbb{P}^{n-1}$ with generically finite fibres,   of degree $d \geq 2$, with $M \subset \pi (Z(\Q))$.  In this setting, the central conjecture is:
\begin{conj}[Projective type II, nonuniform]\label{conj_Serre_projII}
Let $M \subset \Pbb^{n-1}(\Q)$ be a thin set of type II. 
Then 
\[N_{\Pbb^{n-1}}(M,B) \leq C(M) B^{n-1} s(B),\]
    in which $C(M)$ is a constant depending on $M$, and $s(B)$ is a ``small factor,'' which could take the form  $s(B)\ll_\ep B^\ep$ (for any $\ep>0$); or $s(B) \leq (\log B)^{\ga(M)}$ for some $\ga(M) > 0$; or   $s(B) \leq 1$.
\end{conj}
This is stated in \cite[Ch. 13, p. 178]{Ser97} as a reasonable expectation, with small factor  $s(B)\ll (\log B)^{\ga(M)}$, and no discussion of how $C(M)$ depends on $M$. To say that $C(M)$ depends on $M$ is to say it depends on the morphism $\pi: Z \maps \Pbb^{n-1}$, and we will study this dependence in detail later in \S \ref{sec_Serre}-\ref{sec_uniformity_question}. 
 
Serre proved this for $\Pbb^1$, namely showing that $N_{\Pbb^{1}}(M,B) \ll B$; see \cite[\S 9.7. Thm. page 133]{Ser97}.
Broberg \cite{Bro03N} again proved Conjecture \ref{conj_Serre_projII} for $\Pbb^1$, in the stronger form that $N_{\Pbb^{1}}(M,B)\ll_{M,\ep} B^{2/d+\ep}$, in which $d$ denotes the degree of $\pi$. Broberg also proved the conjecture   for $\Pbb^2$ when $d \geq 3$, with small factor $s(B)\ll_\ep B^\ep$ and unspecified dependence $C(M)\ll_M 1$. For $\Pbb^2$ and $d=2$, Broberg proved $N_{\Pbb^{2}}(M,B) \ll_{M,\ep} B^{9/4+\ep}$.

In the case where $\pi$ is a smooth cyclic cover, the conjecture was proved by Heath-Brown and Pierce \cite{HBPie12} for all $n \geq 10$, with mildly weaker bounds for $3 \leq n \leq 9$.

Historically, the first general improvement on the trivial bound $N_{\Pbb^{n-1}}(M,B) \ll_n B^n$, valid for any projective thin set of type II, was Theorem \ref{thm_baseline_affine}, which yields  the baseline upper bound
\beq\label{projII_baseline} N_{\Pbb^{n-1}}(M,B) \leq C(M) B^{n-1/2}(\log B)^{\ga(M)}.
\eeq
This theorem is stated for any thin set in \cite[Ch. 13 Thm. 3]{Ser97}, but it is only nontrivial  for thin sets of type II, by recalling the trivial bound (\ref{projI_trivial_thin_set})   for type I. By an application of Lemma \ref{lemma_affine_version_is_thin}, proving (\ref{projII_baseline}) can be reduced to considering affine thin sets of type II, which we turn to next in \S \ref{sec_affineII}.  

We defer discussing works that improved the baseline bound (\ref{projII_baseline})  (namely \cite{Mun09}, \cite{Bon21}, \cite{BonPie24}, \cite{BPW25x}),     and the recent resolution of Conjecture \ref{conj_Serre_projII} by \cite{BCSSV25x} in nearly all cases, to \S \ref{sec_projII_revisit}, as all these results rely on progress in the analogous affine setting. 

\begin{rem}[Deduction of (\ref{Serre_projI})]
\label{remark_Serre_progress_DGC_via_projII}
Serre made the following observation to prove (\ref{Serre_projI}), as in \cite[p. 178 Thms. 3,4]{Ser97}. Suppose an irreducible variety $V \subsetneq \Pbb^{n-1}$ has $\dimP V=m \leq n-2$ and degree $d \geq 2$. Select coordinates for $\Pbb^{n-1}$ so that projection onto the first $m+1$ coordinates induces a finite morphism $\pi: V \rightarrow \Pbb^{m}$ of degree $d$. The image $\pi(V(\Q))$ is then a projective thin set in $\Pbb^m$ (arguing as in Lemma \ref{lemma_piV_thin_set}). Applying (\ref{projII_baseline}) then shows 
\[ N_{\Pbb^m}(\pi(V(\Q)),B) \ll C(V)B^{(m+1)-1/2}(\log B)^{\ga(V)} = C(V)B^{\dimP V+1/2}(\log B)^{\ga(V)}. 
\]
 Each   fibre of $\pi$ has a finite   number of points, so this implies 
 \[N_\proj (V,B) \leq C(V) B^{\dimP V+1/2}(\log B)^{\ga(V)} .\]   
\end{rem}

We end this section with a technical lemma, which shows that in order to prove Serre's Conjecture \ref{conj_Serre_projII} for a type II thin set $M$ defined according to a suitable morphism $\pi:Z\rightarrow \mathbb{P}^{n-1}$, it is enough to prove it when $\pi$ is quasi-finite. (We will later assume this in Proposition \ref{prop_BCSSV_affineII_projII}.)   We recall from Dfn. 29.20.1 part (3) and the equivalence in Lemma 29.20.10 of \cite[\href{https://stacks.math.columbia.edu/tag/01TC}{Section 01TC}]{StaPro} that for schemes $X,Y$, a morphism $f:X\rightarrow Y$ is quasi-finite precisely when $f$ is locally of finite type, quasi-compact, and has finite fibres.

\begin{lem}\label{lemma_quasi_finite}
Let $Z \subset \Pbb^N$ be an irreducible projective variety of dimension $n-1$ defined over $\Q$.
Let $\pi : Z \maps \Pbb^{n-1}$ with generically finite fibres, be a dominant morphism of degree $d \geq 2$ defined over $\Q$, so that $\pi(Z(\Q))$ is a thin set of type II in $\Pbb^{n-1}$. Then there exists an open dense subset $U \subset Z$ such that $\left. \pi \right|_U:U\rightarrow \mathbb{P}^{n-1}$ is a quasi-finite morphism, and
\[ N_{\mathbb{P}^{n-1}}(\pi(Z(\Q)),B) \leq  N_{\mathbb{P}^{n-1}}(\left. \pi \right|_U(U(\Q)),B)+O_{\pi }(B^{n-1}).\]
\end{lem}

\begin{proof}

 Let $Z$   and $\pi: Z\rightarrow \mathbb{P}^{n-1}$, a generically-finite dominant morphism, be as given. As recalled in \S \ref{sec_morphism_terminology}, $\pi$ is locally of finite type.  Since $\pi$ has generically finite fibres, we can find $V\subset\mathbb{P}^{n-1}$ and $U\subset Z$ open affine, dense subsets such that $\pi_U :=\left. \pi \right|_U:U\rightarrow V$ is finite. In particular, $\#\pi_U^{-1}(v)<\infty$ for every $v\in V$. 
By Lemma 20.15.2 of  \cite[\href{https://stacks.math.columbia.edu/tag/01T0}{Section 01T0}]{StaPro}, since $\pi: Z \maps \Pbb^{n-1}$ is locally of finite type, then for any open sets $U \subset Z$ and $V \subset \Pbb^{n-1}$, $\pi_U: U \maps V$ is locally of finite type.  By the remark immediately before the lemma, to conclude that $\pi_U$ is quasi-finite, it suffices to show that $\pi_U$ is quasi-compact: Since $U\subset Z$ is an affine subset and $Z$ is noetherian then $U$ is noetherian and hence $\pi_{U}$ is quasi compact thanks to \cite[\href{https://stacks.math.columbia.edu/tag/01OU}{Lemma $28.5.8$}]{StaPro}. We may conclude that $\pi_U: U \maps \Pbb^{n-1}$ is quasi-finite.

Now let $Z_0=Z\setminus U$, and consider the image $\pi(Z_0)$. Since $Z_0$ does not contain the generic point of $Z$, it follows that for some  proper closed subset $W \subsetneq \mathbb{P}^{n-1}$, $\overline{\pi(Z_0)}\subset W$ (arguing as in Lemma \ref{lemma_piV_thin_set}). Consequently,
\[
\begin{split}
N_{\mathbb{P}^{n-1}}(\pi(Z(\Q)),B)&= N_{\mathbb{P}^{n-1}}(\pi(U(\Q)),B)+N_{\mathbb{P}^{n-1}}(\pi(Z_0(\Q)),B)\\
&\leq N_{\mathbb{P}^{n-1}}(\pi_U(U(\Q)),B)+N_{\mathrm{proj}}(W,B)\\& = N_{\mathbb{P}^{n-1}}(\pi(U(\Q)),B)+O_{N,\deg W,\dimP W}(B^{n-1}),
\end{split}
\]
where in the last line we used the trivial bound of Lemma \ref{lemma_projI_trivial_bound}.  Since $\deg W$ is inexplicit, we write the last term as $O_{\pi }(B^{n-1}).$
\end{proof}

\subsection{Type II, affine}\label{sec_affineII}
By Lemma \ref{lemma_affine_version_is_thin}, the baseline upper bound
 (\ref{projII_baseline}) is implied by the result that for any affine thin set $M \subset \Abb^n$, 
 \beq\label{affineII_baseline} N_{\Abb^{n}}(M,B) \leq C(M) B^{n-1/2}(\log B)^{\ga(M)}.
\eeq
This appears as \cite[Ch. 13 Thm. 1]{Ser97}.
Recall from Lemma \ref{lemma_typeII_poly_interp} that if  $M \subset \mathbb{A}^{n}(\Z)$ is an affine thin set of type II with corresponding dominant morphism $\pi: Z \maps \Abb^n$, with $\dimA Z=n$ and $\pi$ of degree $ d \geq 2$, such that $M \subset (\pi(Z(\Q)) \cap \Z^n)$, there exists a proper closed subset $V \subsetneq Z$ and an irreducible polynomial $F \in \Z[Y,X_1,\ldots,X_n]$  that is monic in $Y$ and has $\deg_YF=d$, such that 
\[ M \subset \left(\left. \pi \right|_V (V) \cup \{\bx \in \Z^n: \text{$F(Y,\bx)=0$ is solvable over $\Z$}\} \right).\]
Since $\dimA V < n$, it follows that $\dimA \left. \pi \right|_V(V) \leq n-1$ (arguing as in Lemma \ref{lemma_piV_thin_set}), so applying  the trivial bound  from Lemma \ref{lemma_affineI_trivial_bound} shows: 
\[N_{\A^n}(\left. \pi \right|_V(V),B) \ll_{\deg \left. \pi \right|_V, \dimA \left. \pi \right|_V} B^{n-1} \ll_\pi B^{n-1}.\] 
(Here the dependence on $\pi$ is inexplicit since the degree of $V$ is inexplicit.)
Consequently,   the counting function 
\beq\label{Ncov_aff_dfn} 
N^\cov_{\Abb^n}(F,B):=\{\bfx\in [-B,B]^n\cap \Z^n:  \text{$F(Y,\bfx)=0$ is solvable over $\Z$}\}
\eeq
is the essential object to prove (\ref{affineII_baseline}) and to study $N_{\Abb^n}(M,B)$ further.

The following example (for $d=2$) shows that (\ref{affineII_baseline}) cannot essentially be improved (analogous to Example \ref{example_affineI_lower_bound} in the affine type I case).
\begin{example}\label{example_affineII_lower_bound}
For a given integer $d \geq 2$, let
$F(Y,X_1,\ldots,X_n)=Y^d-(X_1+\cdots +X_n).$
Then 
$N^\cov_{\Abb^n}(F,B)   \gg_n B^{n-1/d}$, since 
\[
 N^\cov_{\Abb^n}(F,B) = \sum_{\substack{1 \leq z \leq n B\\ \text{$z$ a $d$th power}}} \sum_{\substack{x_1+\cdots+x_n=z \\ \bx \in [-B,B]^n}} 1 \gg_n B^{1/d} \cdot B^{n-1}.
\]
   \end{example}
  Thus, if the route to bounding $N_{\Pbb^{n-1}}(M,B)$   is via the simplistic construction of a corresponding affine thin set $M'$ via Lemma \ref{lemma_affine_version_is_thin}, in general no further progress below the baseline bound (\ref{projII_baseline}) can be made.
  Fortunately, a different passage from projective thin sets of type II to affine thin sets of type II can be devised (playing a role analogous to Proposition \ref{prop_BHBS06_affineI_projI} in the type I setting).

\begin{prop}\label{prop_BCSSV_affineII_projII}
Let $Z \subset \Pbb^N$ be an irreducible projective variety of dimension $n-1$ defined over $\Q$.  
Let $\pi : Z \maps \Pbb^{n-1}$ be a quasi-finite dominant morphism of degree $d \geq 2$ defined over $\Q$, so that $\pi(Z(\Q))$ is a thin set of type II in $\Pbb^{n-1}$. Then there exists a proper closed subset $V \subsetneq Z$, a positive integer $e$, and an irreducible polynomial $F \in \Z[Y,X_1,\ldots, X_n]$ of the form 
\beq\label{BCSSV_affineII_projII_F_shape}
F(Y,\bX) = Y^d + Y^{d-1}f_1(\bX) + \cdots + f_d(\bX),
\eeq
in which each $f_i \in \Z[X_1,\ldots,X_n]$ is homogeneous of degree $e\cdot i$, such that 
\[ N_{\Pbb^{n-1}} (\pi(Z(\Q)),B) \leq  N^\cov_{\Abb^n}(F,B)+N_\proj(\left.\pi\right|_V(V(\Q)),B).\]
Thus 
\[ N_{\Pbb^{n-1}} (\pi(Z(\Q)),B) \leq  N^\cov_{\Abb^n}(F,B)+O_\pi(B^{n-1}).\]
We will call $F$ the polynomial associated to the thin set $\pi(Z(\Q))$.
 \end{prop}
 This is \cite[Lemma 2.4]{BCSSV25x}, which builds on ideas of \cite[Lemma 3]{Bro03N}. (Compare to the affine analogue in    Lemma \ref{lemma_typeII_poly_interp}.)
 The last conclusion follows similarly to the argument of Lemma \ref{lemma_piV_thin_set}: since $V\subsetneq Z$ is a proper closed subset and $Z$ is irreducible so $\dimP V < \dimP Z$,    the trivial bound  of Lemma \ref{lemma_projI_trivial_bound} for counting points in a projective variety implies:
 \[ N_\proj(\left.\pi\right|_V(V),B) \ll_{\dimP \left.\pi\right|_V(V),\deg \left.\pi \right|_V(V)} B^{\dimP \left.\pi\right|_V(V)+1} \ll_\pi B^{(n-2)+1} \ll_\pi B^{n-1}.\] 
 Since $\deg \left. \pi \right|_V(V)$ is inexplicit, the dependence on $\pi$ is inexplicit.

To use Proposition \ref{prop_BCSSV_affineII_projII} to prove Conjecture \ref{conj_Serre_projII}, it suffices to consider $N^\cov_{\Abb^n}(F,B)$   in the special case that $F$ is absolutely irreducible, by the following lemma (analogous to Lemma \ref{lemma_typeI_geometrically_integral} for type I thin sets).
   \begin{lem}\label{lemma_typeII_geometrically_integral}
Let $F(Y,\bfX)=Y^{d}+f_1(\bfX)Y^{d-1}+\cdots + f_{d}(\bfX)\in\mathbb{Z}[Y,X_1,\dots,X_n]$.
Suppose $F$ is irreducible but not absolutely irreducible. Then $N^\cov_{\Abb^n} (F,B)\ll_{n,\deg F} B^{n-1}$.
\end{lem} 
A special case of this result (in which each $f_i$ is assumed to be homogeneous of degree $e \cdot i$, for a fixed positive integer $e$) is stated in \cite[Lemma 2.5]{BCSSV25x} without  proof; for clarity we provide a proof.
\begin{proof}

For such $F$, we can define a polynomial $\Delta_{F} (\bfX)\in\mathbb{Q}[\bfX]$ such that for every $\bfx\in \mathbb{Z}^{n}$, $\Delta_{F}(\bfx)=
\text{disc}(F(Y,\bfx))$. Hence,
\[
F(Y,\bfx)\text{ is  squarefree}\Longleftrightarrow \Delta_{F}(\bfx)\neq 0.
\]
(Precisely $\Delta_F$ is the discriminant of $F$ over $\mathbb{Q}(\bfX)$, i.e. $\text{Res}\left(F,\frac{\partial F}{\partial Y}\right)$ over $\mathbb{Q}(\bfX)$, and  $\deg \Delta_F(\bX) \ll_{n,\deg F} 1$ as in \cite[Ch. 13 Prop. 1.1]{GKZ08}.)
Under the hypothesis that $F(Y,\bX)$ is irreducible over $\Q$, we claim $\Delta_{F}(\bfX)\not\equiv 0$.  By the Hilbert Irreducibility Theorem (see e.g.  \cite[Ch. 9]{Lan83}), one can find $\bfx\in\mathbb{Z}^{n}$ such that $F(Y,\bfx)$ is irreducible over $\mathbb{Q}$, hence separable, and thus squarefree. For such $\bfx$, $\Delta_F(\bx) \neq 0$, and hence   $ \Delta_{F}(\bfX)\not\equiv 0$, as claimed.
    
    Assume furthermore $F$ is irreducible over $\mathbb{Q}$ but not over $\overline{\mathbb{Q}}$. Over $\overline{\mathbb{Q}}$ we can write
    \[
    F(Y,\bfX)=\prod_{i=1}^{\ell}G_{i}(Y,\bfX),
    \]
    with $\ell \geq 2$,
    where for every $i=1,...,\ell$ we have that  $G_i(Y,\bfX)\in\overline{\mathbb{Q}}[Y,\bfX]\setminus \mathbb{Q}[Y,\bfX]$, and $G_{i}$ irreducible over $\overline{\mathbb{Q}}$. Now assume $\bfx$ is such that $F(Y,\bfx)$ is solvable over $\mathbb{Z}$, so there exists $n \in \mathbb{Z}$ such that $F(n,\bfx)=0$. Since
    \[
    0=F(n,\bfx)=\prod_{i=1}^{\ell}G_{i}(n,\bfx),
    \]
    there exists $k\in\{1,...,\ell\}$ such that $G_{k}(n,\bfx)=0$. Since $G_{k}\not\in\mathbb{Q}[Y,\bfX]$, we can find $\tau\in\text{Gal}(\overline{\mathbb{Q}}/\mathbb{Q})$ such that $\tau (G_{k})\neq G_{k}$. On the other hand, since $\tau (F(Y,\bfX))=F(Y,\bfX)$ it follows that $\tau (G_{k})=G_{j}$ for some $j\in\{1,...,\ell\}\setminus\{k\}$. Hence, we have that
    \[
    0=\tau(G_{k}(n,\bfx))=G_{j}(n,\bfx).
    \]
    Since $n$ is a root of $G_{k}(Y,\bfx)$ and a root of $G_{j}(Y,\bfx)$, it follows that $n$ is a double root of $F(Y,\bfx)$, so that $\Delta_{F}(\bfx)=0$. This shows  that for every $\bfx\in\mathbb{Z}^{n}$:
    \[
    F(Y,\bfx)=0\text{ is solvable over }\mathbb{Z}\Rightarrow \Delta_{F}(\bfx)=0.
    \]
    Hence by the Schwartz-Zippel bound (Lemma \ref{lemma_Schwartz_Zippel}),
    \[
    N^\cov_{\Abb^n} (F,B)\leq\#\{\bfx\in\mathbb{Z}^{n}:\|\bfx\|\leq B,\quad \Delta_{F}(\bfx)=0\}\ll_{n,\deg F} B^{n-1},
    \]
    as claimed.
\end{proof}

\subsubsection{Progress via sieve methods} Several works have proved bounds,   for polynomials $F(Y,\bX)$ of special shape, which approach $N^\cov_{\Abb^n}(F,B) \ll_{F,\ep} B^{n-1+\ep}$ asymptotically as $n \maps \infty$. All of these works are united by the attribute that they used sieve methods: the power sieve, various types of polynomial sieve, and more complicated variants.  For the special shape 
\beq\label{Mun_shape} 
F(Y,\bX)=Y^r - G(\bX)
\eeq
with $r \geq 2$ and $G$ a nonsingular homogeneous form with $\deg G=mr$ for some $m \geq 1$, Munshi proved $N^\cov_{\Abb^n}(F,B) \ll_{F,\ep} B^{n-1+1/n} (\log B)^{\frac{n-1}{n}}$ \cite{Mun09}. (See \cite[Remark 1.1]{Bon21}, which improved the main exponent to $n-1+1/(n+1)$, and work of analogous strength over function fields by Bucur, Cojocaru, Lal\'in and Pierce \cite{BCLP23}.) Heath-Brown and Pierce \cite{HBPie12} improved this, showing that for $G(\bX)$ any polynomial  $d \geq 3$ with nonsingular leading form,  
\beq\label{HBP_result}
N^\cov_{\Abb^n}(F,B)\leq C(F) 
\begin{cases}
    B^{n-1+\frac{n(8-n)+4}{6n+4}}(\log B)^2, & 2 \leq n \leq 8 \\
    B^{n-1+ \frac{1}{2n+10}}(\log B)^2, & n=9\\
    B^{n-1 - \frac{(n-10)}{2n+10}}(\log B)^2, & n \geq 10.
\end{cases}
\eeq
Note that for $n \geq 10$, this upper bound takes the form $N^{\cov}_{\Abb^n}(F,B) \leq C(F) B^{n-1 - \Delta_n}(\log B)^2$ for $\Delta_n \maps 1/2$ as $n \maps \infty$. 
Consider next the shape 
\beq\label{Bon_shape} 
F(Y,\bX) = g(Y) + G(\bX)
\eeq
for $g$ any polynomial of degree $r \geq 2$ in $\Z[Y]$ and  any irreducible form $G \in \Z[X_1,\ldots, X_n]$ of degree $e \geq 2$ such that the projective hypersurface $V(F)$ defined by $F=0$ is nonsingular over $\C$. For such $F$, Bonolis obtained $N^\cov_{\Abb^n}(F,B)\ll_F B^{n-1 + \frac{1}{n+1}}(\log B)^{\frac{n}{n+1}}$ \cite{Bon21}. 
Consider a yet more general setting that does not separate the variables $Y$ and $\bX$: fix an integer $m \geq 2$ and integers $d, e \geq 1$, and let
\beq\label{BonPie_shape}
F(Y,\bX)=Y^{md}+Y^{m(d-1)}f_{1}(\bX)+\ldots+Y^mf_{d-1}(\bX)+f_{d}(\bX),
\eeq
in which for each $1 \leq i \leq d$,  $f_i \in \Z[X_1,\ldots, X_n]$ is a homogeneous form with $\deg f_{i}=m\cdot e\cdot i$, and $f_d \not\con 0$. Under the hypothesis that  the weighted hypersurface $V(F(Y,\bX)) \subset \mathbb{P}(e,1,\ldots,1)$ defined by $F(Y,\bX)=0$ is nonsingular over $\C$, Bonolis and Pierce proved 
\beq\label{BonPie_result}
N^\cov_{\Abb^n}(F,B) \ll_{n,m,d,e}(\log (\|F\|+2))^{e(n)} B^{n-1 + \frac{1}{n+1}} (\log B)^{\frac{n}{n+1}}
\eeq 
for some integer $e(n) \geq 1$ \cite{BonPie24}. If $m=1$, that work obtained the same result, conditional on GRH. (For clarity, note that \cite{BonPie24} stated this result with an implicit constant depending only on $n,m,d,e$ but independent of the maximum coefficient $\|F\|$, but this was corrected in \cite{BonPie24cor} to state that the implicit constant may depend (at most) polylogarithmically on $\|F\|$; thus  (\ref{BonPie_result}) corrects \cite[Thm. 1.1]{BonPie24}.) Most recently, in \cite{BPW25x} the authors of the present paper improved this result by removing the nonsingularity hypothesis: instead, the polynomial $F(Y,\bX)$ need only be absolutely irreducible and satisfy an appropriate ``nondegeneracy hypothesis'' that assures that $F(Y,\bX)$ depends nontrivially on each of $X_1,\ldots,X_n$ after any linear change of variables; we provide details in \S \ref{sec_revisit_affineII}.

As a consequence of Proposition \ref{prop_BCSSV_affineII_projII}, each of the results   $N^\cov_{\Abb^n}(F,B)\ll B^{\al(n)}s(B)$ listed above implies a corresponding result $N_{\Pbb^{n-1}}(M,B) \leq C(M)[B^{n-1} + B^{\al(n)}s(B)]$ for any projective thin set $M$ of type II whose associated polynomial $F$, in the sense of Proposition \ref{prop_BCSSV_affineII_projII}, is of the relevant shape (\ref{Mun_shape}), (\ref{Bon_shape}), or (\ref{BonPie_shape}).

\subsubsection{Affine result to resolve Conjecture \ref{conj_Serre_projII}}
Buggenhout, Cluckers, Salberger, Santens and Vermeulen \cite{BCSSV25x} have recently announced  the resolution of Serre's Conjecture \ref{conj_Serre_projII} in nearly all cases. This is deduced by passing to the affine setting via Proposition \ref{prop_BCSSV_affineII_projII}, and considering any polynomial of the following form. For integers $n \geq 1$, $d \geq 2$, $e\geq 1$, let
\[F_{\mathrm{top}}(Y,\bX) = Y^d +Y^{d-1}f_1(\bX) + \cdots + Yf_{d-1}(\bX) + f_d(\bX)\]
be an absolutely irreducible polynomial,
in which each $f_i$ is homogeneous of degree $e \cdot i$, so that $F_{\mathrm{top}}(Y^e,\bX)$ is homogeneous of degree $de$. Then consider any polynomial of the form
\beq\label{BCSSV_class}
F(Y,\bX) =F_{\mathrm{top}}(Y,\bX) + F_0(Y,\bX)
\eeq
in which $F_0(Y^e,\bX)$ has total degree strictly smaller than $de$. For any   polynomial $F$ of type (\ref{BCSSV_class}), \cite[Thm. 1.2]{BCSSV25x} proves that for some absolute constant $\ga$, and for some $\kappa (n,d)$,
\beq\label{BCSSV_affine} N^\cov_{\Abb^n} (F,B) \ll_{n,d} \|F\|^{\kappa(n,d)} \begin{cases} B^{n-1} & \text{if $d \geq 5$}\\
B^{n-1} (\log B)^{\ga} & \text{if $d =4$}\\
B^{n-1+(2/\sqrt{d}-1)}(\log B)^{\ga} & \text{if $d =2,3$.}
\end{cases}
\eeq

 \subsection{Type II, revisiting the projective setting}\label{sec_projII_revisit}
 
 The work of  Buggenhout, Cluckers, Salberger, Santens and Vermeulen \cite[Thm. 1.1]{BCSSV25x} resolves Serre's question on  counting points in  projective thin sets of type II (Conjecture \ref{conj_Serre_projII}), in nearly all cases:

 \begin{thm}[BCSSV]\label{thm_BCSSV}
Let $M \subset \mathbb{P}^{n-1}(\Q)$ be a thin set of type II, with corresponding integral variety $X$ and  morphism  $\pi: X \maps \Pbb^{n-1}$ that is defined over $\Q$, quasi-finite, dominant and of degree $d\geq2$.
 Then  for some absolute constant $\ga$,
\[ N_{\Pbb^{n-1}}(M,B) \ll_{\pi} \begin{cases} B^{n-1} & \text{if $d \geq 5$}\\
B^{n-1} (\log B)^{\ga} & \text{if $d =4$}\\
B^{n-1+(2/\sqrt{d}-1)}(\log B)^{\ga} & \text{if $d =2,3$.}
\end{cases}
\]
\end{thm} 
 This is deduced by passing  to the affine setting via Proposition \ref{prop_BCSSV_affineII_projII}, and finally applying   (\ref{BCSSV_affine}). 
By Lemma \ref{lemma_quasi_finite}, it suffices to consider the quasi-finite setting, and so the bound in Theorem \ref{thm_BCSSV} holds for any thin set of type II.
It remains an open question to achieve the exponent of Conjecture \ref{conj_Serre_projII} for $d=2,3$, and further understand the dependence of the implicit constant on $\pi$.

  \begin{rem}[Inexplicit bounds]\label{remark_BCSSSV_inexplicit}
The $O_\pi(B^{n-1})$ term in the conclusion of Lemma \ref{lemma_quasi_finite} has inexplicit dependence on $\pi$. This arises because the construction of the proper closed subset $W \subset \Pbb^{n-1}$ in the proof of that lemma has inexplicit degree. (See also \cite[Thm. 1.1]{BCSSV25x}, which alternatively uses Zariski's Main Theorem in order to reduce to consideration of finite covers, also leading to an inexplicit conclusion.) Similarly, the degree of $V$ in Proposition \ref{prop_BCSSV_affineII_projII} is inexplicit, leading to an $O_\pi(B^{n-1})$ term in that conclusion. Thus, even though the affine result (\ref{BCSSV_affine}) explicitly controls dependence on $\|F\|$, the application in Theorem \ref{thm_BCSSV} has inexplicit dependence on $\pi$. 
    
\end{rem}
   
 \subsection{Type II, revisiting the affine setting}\label{sec_revisit_affineII}
By  Example \ref{example_affineII_lower_bound}, in general the baseline affine upper bound (\ref{affineII_baseline}) cannot   be essentially improved in general. However, parallel to the type I setting of \S \ref{sec_revisit_affineI}, it is reasonable to ask   for the largest class of affine thin sets of type II for which the baseline   upper bound can be improved to exponent $n-1$, or analogously, the largest class of ``allowable'' polynomials $F(Y,\bX)$ for which one can show 
\beq\label{BPW_goal} 
N^\cov_{\Abb^n}(F,B)\leq C(F) B^{n-1}s(B)
\eeq 
for a constant $C(F)$ and small factor $s(B)$. By (\ref{BCSSV_affine}), any polynomial of the shape (\ref{BCSSV_class}) belongs to this class, for $d \geq 4$.
Moreover, the results in (\ref{HBP_result}) suggest it is an interesting question to characterize affine thin sets of type II for which an even better upper bound holds. Finally, it is also interesting  to study uniformity questions related to the factor $C(F)$. 

Our recent work \cite{BPW25x} advanced  the first line of inquiry. There, we described a large class of  polynomials $F(Y,\bX) \in \Z[Y,X_1,\ldots, X_n]$, for which
\beq\label{BPW_strength} 
N^\cov_{\Abb^n}(F,B) \ll_\ep C(F)B^{n-1+\delta_n+\ep}, 
\eeq
for  $\delta_n=1/(n+1)$, so that this result approaches (\ref{BPW_goal}) as $n \maps \infty$.  Let us state what makes a polynomial ``allowable'' more precisely, using the terminology of \cite{BPW25x}. First,  consider absolutely irreducible polynomials of the form 
\beq\label{BPW_class}
F(Y,\bX) =Y^{md} +Y^{m(d-1)}f_1(\bX) + \cdots + Y^m f_{d-1}(\bX)+f_d(\bX),
\eeq 
for some integers $m,d \geq 1$ with $md \geq 2$.   No homogeneity condition or degree condition is placed on the polynomials $f_j$, so that this class is larger than the class (\ref{BCSSV_class}).
 
Let $L$ be a finite (nontrivial) extension of $\Q(X_1,...,X_n)=\Q(\bX)$. We say   $L$ is an  $n$-genuine extension of $\Q(\bX)$ if for every $G(Y,\bX)\in \Z[Y,\bX]$ such that $L = \Q(\bX)[Y]/(G(Y,\bX)),$
$G(Y,\bX)$ has nonzero degree in each of $X_1,\ldots,X_n$. 
We say an absolutely irreducible polynomial $F(Y,\bX)\in \Z[Y,\bX]$ that is monic in $Y$ and has $\deg_Y F \geq 2$ is  \emph{$(1,n)$-allowable} if for every linear change of variables $\sigma\in \GL_n(\Q)$, the polynomial  $F_\sigma(Y,X_1,...,X_n) := F(Y,\sigma(\bX))$ has the property that $\Q(\bX)[Y]/(F_\sig(Y,\bX))$ is an $n$-genuine extension.  
For every polynomial $F$ in the class (\ref{BPW_class}) for $m \geq 2$ that is $(1,n)$-allowable, \cite[Thm. 1.2]{BPW25x} proves (\ref{BPW_strength}) with $C(F)\ll_{n,m,\deg F}  \|F\|^{\kappa(n)}$ depending at most polynomially on $\|F\|$. (If $m=1$, the same result holds if GRH is assumed.)
For each fixed $m \geq 1$, $(1,n)$-allowable polynomials are generic among polynomials of the form (\ref{BPW_class}), by \cite[Cor. 3.2]{BPW25x}.

The dependence on $\|F\|$ can be reduced if a stronger nondegeneracy assumption is made. 
 Precisely, we say that a finite (nontrivial) extension $L$ of $\Q(X_1,...,X_n)=\Q(\bX)$ is a strongly $n$-genuine extension of $\Q(\bX)$ if for all subextensions $L'$ satisfying  $\Q(\bX) \subsetneq L' \subset L,$
    $L'$ is an $n$-genuine extension. 
We say that an absolutely irreducible polynomial $F(Y,\bX) \in \Z[Y,\bX]$ that is monic in $Y$ and has $\deg_Y F \geq 2$ is   \emph{strongly $(1,n)$-allowable} if  for every linear change of variables $\sigma\in \GL_n(\Q)$, the polynomial   $F_\sigma(Y,X_1,...,X_n) := F(Y,\sigma(\bX))$ has the property that $\Q(\bX)[Y]/(F_\sig(Y,\bX))$ is a strongly $n$-genuine extension.  
For $F$ in the class (\ref{BPW_class}) for $m \geq 2$ that is strongly $(1,n)$-allowable,  \cite[Thm. 1.3]{BPW25x} proves (\ref{BPW_strength}) with $C(F)\ll_{n,m,\deg F} (\log (\|F\|+2))^{\kappa(n)}$. (If $m=1$, the same result holds if GRH is assumed.)  This relies on uniformity results for trace functions recently obtained in \cite{BKW25x}. For each fixed $m \geq 1$, strongly $(1,n)$-allowable polynomials are generic among polynomials of the form (\ref{BPW_class}), by \cite[Cor. 3.2]{BPW25x}.

  By different methods, for any polynomial $F$ that is $(1,n)$-allowable but not strongly $(1,n)$-allowable, \cite[Thm. 1.4]{BPW25x} attains the stronger bound  (\ref{BPW_goal}), but now with $C(F)$ depending on $\|F\|$ in an unspecified way.  It remains an interesting open problem to characterize those polynomials for which (\ref{BPW_goal}) holds.

\begin{rem}
  We remark on an underlying similarity between the notion of nondegeneracy encapsulated by  $(1,n)$-allowability here, and the notion  in \S \ref{sec_revisit_affineI} of NCC (not cylindrical over a curve) for a polynomial $f$. Here, requiring a polynomial $F(Y,\bX)$ to be $(1,n)$-allowable excludes the possibility that a $\GL_n(\Q)$ change of variables reduces the number of variables; compare to Remark \ref{remark_NCC}.
  \end{rem}
This completes our survey of recent results on counting  points in thin sets.

\section{Revisiting Serre's large sieve argument}\label{sec_Serre}

 In the remainder of the paper, we focus on questions of dependence and uniformity, in the setting of affine thin sets of type II.
 In this section, we show that Serre's original large sieve argument   to prove the baseline bound in Theorem \ref{thm_baseline_affine}, in particular for affine thin sets of type II, can be refined to control the dependence of $C(M)$, in some cases. Additionally, we  clarify the role of Dedekind zeta functions and GRH, building on \cite[\S3.2]{Ser97}. Our main result in this section is:
\begin{thm} \label{thm_Serre}
Let $n \geq 2$ be fixed.
Let $F(Y,X_1,\ldots,X_n) \in \Z[Y,X_1,\ldots,X_n]$ of total degree $D$ be absolutely irreducible   and of the form 
\beq\label{F_dfn_Serre}
F(Y,X_1,...,X_n) = Y^{md} + Y^{m(d-1)} f_1(\bX) +\cdots + Y^mf_{d-1}(\bX)+ f_d(\bX),
\eeq
for a given integer $m \geq 2$.  
Then for some $c(n) \geq 1$,
\[N^\cov_{\Abb^n}(F,B)\ll_{n,D}  (\log (\|F\|+2))^{c(n)} B^{n-1/2}\log (B+2) \qquad \text{for all $B \geq 1$.}
\]
   If $m=1$ and $\deg_Y F \geq 2$, the same result holds, conditional on GRH for a single Dedekind zeta function. 
\end{thm}
 Serre's original argument does not specify the dependence on $\|F\|$, but it improves the $\log B$ factor to $(\log B)^{\ga(F)}$ for some $\ga(F)<1$.  Including those additional steps from Serre's argument with our presentation of the $m=1$ case  will lead to this same log factor in Theorem \ref{thm_Serre} (see \S \ref{sec_m1} and Remark \ref{remark_refined_log}).

The proof relies on both the large sieve and the Lang-Weil bound, which we recall in the affine form stated in \cite[Lemma 2.4]{BPW25x}: 
\begin{lem}\label{lemma_Lang_Weil_cor}
Let $F\in \Z[X_1,\ldots,X_k]$ have degree $d$ (not necessarily homogeneous)  and let $p$ be a prime such that the affine variety $\{F(\bX)=0\}\subset \mathbb{A}^k(\F_p)$ of dimension $k-1$ is irreducible over $\overline{\F}_p$. Then 
\[ \# \{(x_1,\ldots,x_k) \in \F_p^k: F(\bx)=0\} = p^{k-1} + O_{k,d}(p^{k-1-1/2}).\]
\end{lem}
We also require a standard consequence of Noether's Lemma, which we quote from \cite[Lemma 2.9]{BPW25x}.
\begin{lem}\label{lemma_Bertini_Noether}
Let $F(\bX) \in \Z[X_1,\ldots,X_n]$ have degree  $D \geq 1$. If $F$ is  irreducible over $\overline{\Q}$, then there exists a finite set $\mathcal{E}$ of exceptional primes with $|\mathcal{E}| \ll_{n,D} \log \|F\|/ \log \log \|F\|$  such that for $p \not\in \mathcal{E}$, the reduction of $F$ modulo $p$ is again of degree $D$ and is irreducible over $\overline{\F}_p$. 
\end{lem}

Now we turn to the proof of Theorem \ref{thm_Serre}.
 Initially, let $m \geq 1$ be fixed. For simplicity, we assume $\|F\|\geq 3$, $B \geq 3$ in what follows. As Serre states, a formulation of the large sieve in \cite[\S12.1]{Ser97} (specialized to the case $K=\Q$) shows that 
    \beq\label{NFBQ}
    N^\cov_{\Abb^{n}}(F,B) \leq 2^n (B^{n} + Q^{2n})/L(Q)
    \eeq
    where $L(Q)$ is a sum over squarefree integers defined by
    \beq\label{LQ_dfn} L(Q) = \sum_{\substack{q \leq Q, \\ \text{squarefree}}} \prod_{p \mid q} \frac{\om_p}{1 - \om_p},
    \eeq
    in which $\om_p$ is defined by
    \beq\label{omp_dfn}
    1 - \om_p = \frac{|N_p(F)|}{p^n},\eeq 
    and 
    \[ N_p(F)  := \#\{ (x_1,\ldots,x_n) \in \F_p^n  :  \exists y \in \F_p, F(y,x_1,\ldots,x_n) =0 \; \text{in $\F_p$} \}.\]
  Choose $Q=B^{1/2}$ to balance the two main terms in (\ref{NFBQ});  to prove the proposition, it suffices to show that $L(Q) \gg Q/ \log Q$, and to quantify any dependence on $\|F\|$. 
 By positivity, we may bound $L(Q)$ below by restricting the sum over $q \leq Q$ to any convenient nonempty set of primes $p \in [Q/2, Q]$; the task is then to produce a lower bound for $\om_p$ for each prime in this set.

 We make a preliminary observation. For any prime $p$, define 
 \[ M_p(F) := \#\{ (y,x_1,\ldots,x_n) \in \F_p^{n+1}  :   F(y,x_1,\ldots,x_n) =0 \; \text{in $\F_p$} \}.\]
    By Lemma \ref{lemma_Bertini_Noether}, aside from a set $\mathcal{E}$ of at most $\ll_{n,D} \log \|F\|/\log \log \|F\|$ exceptional primes, the reduction of $F$ modulo $p$ is  irreducible over $\overline{\F}_p$. For each $p \not\in \mathcal{E}$,   the Lang-Weil bound (Lemma \ref{lemma_Lang_Weil_cor}) implies
$M_p(F) = p^n + O_{n,D}(p^{n-1/2}).$
 In particular, for all $p \gg_{n,D} 1$ with $p \not\in \mathcal{E}$, $M_p(F) \geq 1$, so that $N_p(F) \geq 1$ and hence $\frac{\om_p}{1-\om_p}$ is finite. In the main arguments that  follow, we will only consider $B$ sufficiently large  so that this lower bound holds for  a positive proportion of primes $p \in [Q/2,Q]$, when $Q=B^{1/2}$.
    
\subsection{The case $m \geq 2$} 
We first restrict our attention to the case in which $m \geq 2$ in the assumed shape (\ref{F_dfn_Serre}) of the polynomial $F$. 
We can prove a lower bound for $L(Q)$ by only considering primes $p \in[Q/2,Q]$ such that $p \con 1 \modd{m}$.
For each such prime, we now focus on proving a  lower bound for $\om_p$.  Write $N_p(F)= N_p'(F) +N_p''(F)$ where the first term is the contribution from those $(x_1,\ldots,x_n)$ such that $f_d(\bx)=0$, and the second term is those such that $f_d(\bx)\neq 0$, in which case any solution $y$ to $F(Y,x_1,\ldots,x_n)=0$ must have $y \neq 0$.  By the trivial bound (Lemma \ref{lemma_Schwartz_Zippel} over $\F_p$), $N_p'(F) \ll_{D} p^{n-1}$. 
On the other hand, if $f_d(\bx) \neq 0$, $y$ is a nonzero solution in $\F_p$ to $F(Y,x_1,\ldots,x_n)=0$   and $\ga$ is a primitive $m$-th root of unity in $\F_p$, then $y, \ga y,\ldots, \ga^{m-1} y$ provide $m$  distinct solutions to the equation. This yields the essential relation  
    $M_p(F) \geq m N_p''(F)$, with $M_p(F)$ as defined above. With the exceptional set $\mathcal{E}$ as defined above, for each $p \not\in \mathcal{E}$ with $p \con 1 \modd{m}$, 
\beq\label{NMm}
N_p''(F) \leq \frac{1}{m}M_p(F) = \frac{1}{m}(p^n + O_{n,D}(p^{n-1/2})).
    \eeq
    For such $p$, by construction 
    \[\om_p = 1-\frac{|N_p'(F)|}{p^n} -\frac{|N_p''(F)|}{p^n}\geq  1- \frac{1}{m} + O_{D}(p^{-1})+ O_{n,m,D}(p^{-1/2}) \gg_m 1 ,
    \]
    as long as $p \gg_{n,m,D} 1$.
    By the Siegel-Walfisz theorem for primes in arithmetic progressions, for all $Q \gg_m 1$ sufficiently large, 
    $\#\{ p \in [Q/2,Q]: p \con 1 \modd{m}\}\gg_m Q/\log Q$.  We can guarantee that  at least half, say, of these lie outside the finite exceptional set $\mathcal{E}$ as long as  $Q \gg_{n,m,D} (\log \|F\|)(\log \log \|F\|)$.
      Consequently, by positivity, for such $Q$,
    \begin{align*} 
    L(Q) &\geq \sum_{\substack{p \leq Q \; \text{prime}\\ p \con 1 \modd{m}}} \frac{\om_p}{1 - \om_p}
    \gg_m \sum_{\substack{p \in [Q/2,Q] \setminus \mathcal{E}\\ p \con 1 \modd{m}}}  \om_p  \\  
    &\gg_m \#\{ p \in [Q/2,Q]\setminus \mathcal{E}: p \con 1 \modd{m} \} \gg_m Q/\log Q.
    \end{align*}
  This proves the claimed lower bound for $L(Q)$ for all  $Q \gg_{n,D}  (\log \|F\|)(\log \log \|F\|)$, and hence $N^\cov_{\Abb^n}(F,B)\ll_{n,D} B^{n-1/2}\log B$ has been proved if  $B \gg_{n,D}  (\log \|F\|)^2(\log \log \|F\|)^2$, and so certainly for $B \gg_{n,D} (\log \|F\|)^3$, say.   On the other hand, for $B \ll (\log \|F\|)^3$, we may apply the trivial bound 
  $N^\cov_{\Abb^n}(F,B) \ll_n B^n \ll (\log \|F\|)^{3n}.$ In total, 
\[N^\cov_{\Abb^{n}}(F,B)\ll_{n,D} (\log \|F\|)^{c(n)} B^{n-1/2}\log B \qquad \text{for all $B \geq 1$,}
\]
 in which  $c(n)=3n$ suffices, and the theorem is proved in the case $m=2$. 
  
\subsection{The case $m=1$}\label{sec_m1}
In the case that $m=1$ we must impose that $d \geq 2$ in (\ref{F_dfn_Serre}), so the morphism $\pi: V(F(Y,X_1,\ldots,X_n)) \rightarrow (X_1,\ldots,X_n)$ has degree at least 2.   Now if $m=1$ the relation (\ref{NMm}) no longer suffices to imply $\omega_p \gg_m 1$. Instead,  the argument of Serre in \cite[\S 13.3]{Ser97} proceeds via \cite[\S 13.1 Thm. 5]{Ser97}, which we record with additional attention to uniformity:

    \begin{lem}\label{lemma_Serre_splits_completely}
  Let $F(Y,X_1,\ldots,X_n) \in \Z[Y,X_1,\ldots,X_n]$ be an absolutely irreducible polynomial of total degree $D>1$.       There is a constant $c=c(n,D)$   with $0<c <1$, an exceptional set $\calE$ with $|\calE|\ll_{n,\deg F} \log\|F\|/\log \log \|F\|$, and a   finite Galois extension $K/\Q$ such that if $p$ splits completely in $K$ and $p\not\in \calE$, then $$|N_p(F)|=\#\{\bx\in \F_p^n: \exists y \in \F_p, F(y,\bx)=0\} \leq c p^n + O_{n,D}(p^{n-1/2}).$$
  Moreover, $2 \leq [K:\Q]\ll_{n,D} 1$ and  $\log |\mathrm{Disc}K/\Q| \ll_{n,D} \log\|F\| $. 
    \end{lem}

   \begin{rem}\label{remark_refined_log}     In fact, Serre shows one may take $c =1 - 1/r!$ where $2\leq r \leq [K:\Q]$. (In the notation of the argument presented below, $r = [K(Z):K(X_1,...,X_n)] \leq  [K:\Q]$.) 
Let $c^*=1/c-1$.
Serre presents an argument to improve the factor $\log B$ in Theorem \ref{thm_Serre} to $(\log B)^{\ga(F)}$ for  $\ga(F)=1-\lambda(F)$, in which $\lambda(F):=c^*/[K:\Q]>0$; see \cite[p. 182-183]{Ser97}. 
\end{rem}

      \begin{proof}
      We elaborate on Serre's proof in order to quantify dependence on $\|F\|$.
        Let $Z = V(F(Y,X_1,...,X_n))$ denote the vanishing set of $F$, and let $D$ denote the total degree of $F$. Then define $K(Z)$ to be the function field of $Z$; this a finite field extension of $\Q(X_1,...,X_n)$ of degree $\ll_{n,D}1$. Take $\Omega$ to be the Galois closure of $K(Z)$ over $\Q(X_1,\ldots, X_n)$. Define $K$ to be the integral closure of $\Q$ in $\Omega$ and $Z_\Omega$ the normalization of $Z$ in $\Omega$. 
         In particular, $2 \leq [K:\Q]\ll_{n,D}1$ since $\deg F\geq 2.$

        Now, we want to take an open subscheme $U\subset \A^n$ such that $U_Z/U$ and $U_{Z_\Omega}/U$ are finite and \'etale. We write $U_{Z_\Omega}$ as $U_{\Omega}$ for simplicity. In particular, viewing $F(Y,X_1,...,X_n)$ as a polynomial in $Y$, we have that $U = \{\disc(F(Y,X_1,...,X_n))\neq 0\}\subset \A^n.$ Now we note that the defining polynomial equations of $U$ will have degree bounded by $n$ and $D$. As a result, $U_Z$ and $U_\Omega$ will also be defined by polynomials of degree bounded by $n$ and $D$. Additionally, the field of constants of $U_\Omega$ is $K$, so $U_\Omega$ is absolutely irreducible over $K$. 

        Let $p$ be a prime; by Noether's Lemma \ref{lemma_Bertini_Noether}, apart from an exceptional set $\calE$ with   $|\calE|\ll_{n,D}\log\|F\|/ \log \log \|F\|$, $U_{\Omega,p} = U_{\Omega}(\F_p)$ will be absolutely irreducible. As a consequence of Lang-Weil (Lemma \ref{lemma_Lang_Weil_cor}),
        $$|U_{\Omega}(\F_p)| = p^n + O_{n,D}(p^{n-1/2}),$$
        where furthermore the explicit dependencies on $n$ and $D$ can be described by $$|U_{\Omega}(\F_p)| = \sum_{i} \pm \lambda_i(p).$$
      Here $\lambda_i(p)$ are the eigenvalues of the Frobenius map at $p$ acting on $\ell$-adic cohomology with compact support, $\lambda_i(p) = \pm p^n$ for exactly one value of $i$, $|\lambda_i(p)|\leq p^{n-1/2}$ for all other values of $i$, and the number of such eigenvalues is $\leq C(n,D)$; the constant $C(n,D)$ is described in \cite[Theorem 1]{Bom78}.
       Now, following identically the rest of the proof given by Serre,   $$\#\{\bx\in \F_p^n: \exists y \in \F_p, F(y,\bx)=0\} \leq (1- \frac{1}{[K:\Q]!}) p^n + O_{n,D}(p^{n-1/2}).$$

     Finally, we claim that $\log|\textrm{Disc}(K/\Q)|\ll_{n,D}\log \|F\|.$ To see this, consider the discriminant $\textrm{Disc}(\Omega/\Q(X_1,...,X_n))$ as an algebraic function in $X_1,...,X_n.$ If $\alpha_1(X_1,...,X_n), ..., \alpha_{D}(X_1,...,X_n)$ are the roots of $F(Y,X_1,...,X_n)$, then $\Omega = \Q(X_1,...,X_n)(\alpha_1(X_1,...,X_n),...,\alpha_{D}(X_1,...,X_n)).$ Consequently   $\textrm{Disc}(\Omega/\Q(X_1,...,X_n))$ is a polynomial of degree $\ll_{n,D} 1$ in terms of  the $\alpha_i(X_1,...,X_n)$. Additionally, these algebraic functions $\alpha_i(X_1,...,X_n)$ satisfy that for any choice of  $(x_1,...,x_n)\in \Q^n$ such that $\alpha_i(x_1,...,x_n)$ is a well-defined value in $\C$, we have that 
     \[\log \|\alpha_i(x_1,...,x_n)\|\ll_{n,\deg F} \log\|F\| +\log \|(x_1,...,x_n)\|.\] 
     (This follows from the fact that the coefficients of the polynomial $F(Y,X_1,...,X_n)$ will be symmetric polynomials of the roots $\alpha_i(X_1,...,X_n)$.) Thus for any $\bx\in \Q^n$, $$\log|\textrm{Disc}(\Omega/\Q(X_1,...,X_n))(x_1,...,x_n)| \ll_{n,D} \log\|F\| + \log\|(x_1,...,x_n)\|.$$
    Since $K$ is the integral closure of $\Q$ in $\Omega$, this implies that $\log|\textrm{Disc}(K/\Q)| \leq_{n,D} \log\|F\|.$
        
    \end{proof}

As a consequence of the lemma, for all  $p \not\in \mathcal{E}$, the quantity $\om_p$ defined in (\ref{omp_dfn}) satisfies
 \[ \om_p \geq 1-c +O_{n,D}(p^{-1/2}) \geq c' >0\] 
 where $c'=c'(n,D)$, as long as $p \gg_{n,D} 1$. Then in the expression (\ref{LQ_dfn}) for $L(Q)$, by restricting the sum to such primes, as long as $Q \gg_{n,D} 1$, 
   \beq\label{LQ_GRH} L(Q) \geq \sum_{\substack{ p \in [Q/2,Q]\setminus \mathcal{E}\\ \text{p splits completely in $K$}}} \frac{\om_p}{1 - \om_p}\geq c' \#\{ p \in [Q/2,Q]\setminus \mathcal{E}, \text{$p$ splits completely in $K$}\}.\eeq
   Let $G = \mathrm{Gal}(K/\Q)$ (in the notation of Lemma \ref{lemma_Serre_splits_completely}).
 By the Chebotarev density theorem in the form of \cite[Thm. 1.1.]{LagOdl77} as refined in \cite[Thm. 4]{Ser81}, if the Dedekind zeta function $\zeta_K(s)$ of the finite Galois extension $K/\Q$ constructed in Lemma \ref{lemma_Serre_splits_completely} satisfies the Riemann Hypothesis, then for every $Q \geq 2$,
 \beq\label{Cheb_GRH} |\# \{ p \leq Q: \text{$p$ splits completely in $K$}\}  - \frac{1}{|G|} \frac{Q}{\log Q}| \ll_{n,|G|} Q^{1/2} (\log Q + \log |\mathrm{Disc}(K/\Q)|). 
 \eeq
By Lemma \ref{lemma_Serre_splits_completely}, $|G|\ll_{n,D} 1$ and  $\log |\mathrm{Disc}(K/\Q)|\ll_{n,D} \log \|F\|.$
Thus as long as $Q \gg_{n,D} (\log \|F\|)^{2+\ep_0}$ for some $\ep_0>0$, both $|\mathcal{E}|$ and the remainder term above can be made at most $(1/4)|G|^{-1}Q/\log Q$ in absolute value. Applying this in (\ref{LQ_GRH}) shows that $L(Q) \gg_{n, D} Q/\log Q$ for $Q \gg_{n,D} (\log \|F\|)^{2+\ep_0}$. Inserting this in (\ref{NFBQ}) and recalling $Q = B^{1/2}$ shows that 
$$N^\cov_{\Abb^n}(F,B) \ll_{n,D} B^{n-1/2}\log B$$ for
 $B \gg_{n,D} (\log\|F\|)^{5}$, say. As before in the case $m \geq 2$, this estimate can be rephrased as the conclusion of Theorem \ref{thm_Serre} by including a factor of $(\log \|F\|)^{5n}$, say.

\subsection{The role of GRH}\label{remark_GRH_number_of_apps}
Let $F(Y,\bX)$ be as in (\ref{F_dfn_Serre}) for a given integer $m \geq 1$.
Both Serre's large sieve approach to bounding $N^\cov_{\Abb^n}(F,B)$ and the polynomial sieve approach that we have developed in \cite{BonPie24} and \cite{BPW25x}, depend on counting  solutions to $F(Y,\bx)$ in $\F_p$, and therefore depend on finding sufficiently many primes with a certain property. When $m=1$ (so that $F(Y,\bX)$ is simply monic in  $Y$), this introduces a role for GRH in both proofs, but in somewhat different ways.

The case of $m=1$ in Theorem \ref{thm_Serre}, as argued here, only requires assuming GRH for one Dedekind zeta function (of the  Galois extension $K/\Q$ constructed in Lemma \ref{lemma_Serre_splits_completely}).
Instead of the full strength of GRH in (\ref{Cheb_GRH}) it would suffice to know the weaker property that the right-hand side of (\ref{Cheb_GRH}) can be replaced by $c_{|G|} Q/\log Q$ for a sufficiently small effective constant $c_{|G|}$. (This is equivalent to proving effective control of the possible Siegel zero of the Dedekind zeta function $\zeta_K(s)$ in the error term of the Chebotarev density theorem in this case; see e.g. \cite[Eqn. (3.6)]{BonPie24}.
It would even be sufficient to prove this in some average sense, over primes $p \in  [Q/2,Q]$.)

In comparison, the conditional version of the polynomial sieve stated in \cite[Lemma 6.1]{BPW25x} when $m=1$ assumes that GRH holds for a large collection of Dedekind zeta functions (each associated to some $\bfk \in [-2B,2B]^n$, in the proof provided in \cite[\S 3.2]{BonPie24}). As $B \rightarrow \infty$, this formulation of the polynomial sieve assumes GRH for the Dedekind zeta functions of potentially infinitely many fields. See \cite[Remark 7.3]{BPW25x} for an alternative formulation of the conditional polynomial sieve (when $m=1$) that replaces the assumption of GRH with an average statement, also still unknown.

\section{An unconditional polylog   estimate}
\label{sec_polylog}

In this section, we consider polynomials that have  constant-leading-coefficient in $Y$. That is to say,
\[F(Y,\bX) = c_DY^D +Y^{D-1} f_{1}(\bX) +\cdots +f_D(\bX),\]
for an integer $c_D\neq 0$; this of course includes all  polynomials that are monic in $Y$. 
 We now describe an unconditional method to bound $N^\cov_{\Abb^n}(F,B)$ as in the baseline bound (\ref{NFB_nonuniform_intro}), with  at most polylog dependence on $\|F\|$.
\begin{thm}\label{thm_polylog}
Fix an integer $n \geq 1$. Let $F(Y,X_{1},...,X_{n})\in\mathbb{Z}[Y,X_{1},...,X_{n}]$ be irreducible in $\mathbb{Q}[Y,X_{1},...,X_{n}]$ of total degree $D$, with   $\deg_{Y}(F)\geq 2$, and assume that  $F$ has constant-leading-coefficient in $Y$. Then  for some $e(n) \geq 1$,
\[
N^\cov_{\Abb^n}(F,B)\ll_{n,D}  (\log (\|F\|+2))^{e(n)} B^{n-\frac{1}{2}}(\log (B+2))^{e(n)}\qquad \text{for all $B \geq 1$.}
\] 
\end{thm}
Note that to prove the theorem, we may assume that $\deg_{X_i}F(Y,\bX) \geq 1$ for some $i=1,\ldots,n$. For indeed, suppose $F(Y,\bX)$ is irreducible over $\Q$ and $\deg_Y(F) \geq 2$ while  $\deg_{X_{i}}F=0$ for all $i=1,\ldots,n$. Then $F(Y,\bx)=F(Y,\bX)$ for every $\bfx \in \Z^n$, so that $F(Y,\bfx)$ is never solvable over $\ZZ$ since by hypothesis $F$ is irreducible over $\mathbb{Q}$, and consequently $N^\cov_{\Abb^n}(F,B)=0$.

The proof of the theorem originates in \cite[Thm. 2.13]{BPW25x}, but we report it here to clarify how it inherits (at most) polylogarithmic dependence on $\|F\|$ from a recent result of Cluckers et al, on effective Hilbert's Irreducibility Theorem, obtained via the determinant method.
\begin{thm}[{\cite[Theorem 1.10]{CDHNV23x}}]\label{thm_CDHNV}
Let $F(T,X_{1},...,X_{m})\in\mathbb{Z}[T,X_{1},...,X_{m}]$ be irreducible in $\mathbb{Q}[T,X_{1},...,X_{m}]$ with  $d=\deg_{\bX}(F)\geq 1$. 
 For each integer $B \geq 1$ define
 \[ R_T(F,B) = \{ t \in \Z \cap [-B,B]: \text{$F(t,X_1,\ldots,X_m)$ is reducible in $\Q[X_1,\ldots,X_m]$}\}.\]
 Then there is a constant $h(m)>0$ depending only on $m$ such that 
 for all $B \geq 1$, 
 \[ R_T(F,B) \ll_m 2^{h(m)d}d^{h(m)}(\log (\|F\|+2))^{h(m)}B^{1/2}.
 \]
\end{thm}
 
Note that  if $\deg_T(F) =0$, then the hypothesis that $F(T,X_1,\ldots,X_m)=F(X_1,\ldots,X_m)$ is irreducible over $\Q$ is independent of specializing $T=t$ to any value $t$. In this case the set $R_T(F,B)$ is empty for all $B$, and the theorem is vacuously true. Note also that the theorem  does not require $F$ to be monic, or even constant-leading-coefficient, in $T$.
  
To deduce Theorem \ref{thm_polylog}, we start by observing that, since $\deg_{Y}(F)\geq 2$,
\begin{multline}\label{F_solvable_reducible}
\#\{\bfx\in[-B,B]^n \cap \mathbb{Z}^{n}:\text{$F(Y,\bfx)=0$ is solvable over $\Z$}\}\\
\leq \#\{\bfx\in [-B,B]^n \cap \mathbb{Z}^{n}: \text{$F(Y,\bfx)=0$ is reducible in $\Q[Y]$}\}.
\end{multline}

It thus suffices to prove:
\begin{prop}
    Fix an integer $n \geq 1$. Let $F(Y,X_{1},...,X_{n})\in\mathbb{Z}[Y,X_{1},...,X_{n}]$ be irreducible in $\mathbb{Q}[Y,X_{1},...,X_{n}]$ of total degree $D$,  with degree $\deg_{Y}(F)\geq 1$, and assume that  $F$ has constant-leading-coefficient in $Y$. 
    Then for some $e(n) \geq 1$,
    \begin{multline*}
         \#\{\bfx\in [-B,B]^n \cap \mathbb{Z}^{n}: \text{$F(Y,\bfx)=0$ is reducible in $\Q[Y]$}\} \\
         \ll_{n,D}(\log (\|F\|+2))^{e(n)}B^{n-1/2}(\log B)^{e(n)}.
             \end{multline*} 
\end{prop}

Similar to various remarks above, for each $n \geq 1$ it suffices to prove this when $\deg_{\bX}F \geq 1$, since if $F$ is independent of $\bX$ the set on the left-hand side must be empty.

 If $n=1$, then we may apply  Theorem \ref{thm_CDHNV}   to $F(Y,X_1)$ with $t$ playing the role of $x_1$. (Since $F$ is constant-leading-coefficient in $Y$, observe that $\deg_Y F \geq 1$, as required.)
Consequently,
 \[ \#\{x_1\in [-B,B] \cap \mathbb{Z}: \text{$F(Y,x_1)=0$ is reducible in $\Q[Y]$}\} \ll_{D}( \log (\|F\|+2))^{h(1)}B^{1/2}.\]
 From now on, let $n \geq 2$, and we proceed by induction on the number  $n$ of variables. Let us now assume that   for any polynomial $G\in\mathbb{Z}[Y,X_{1},...,X_{n-1}]$ that is irreducible in $\mathbb{Q}[Y,X_{1},...,X_{n-1}]$ and of   $\deg_Y(G) \geq 1$ with constant-leading-coefficient in $Y$, the following inductive hypothesis holds:
  \begin{multline*}
       \#\{\bx' \in [-B,B]^{n-1} \cap \mathbb{Z}^{n-1}: \text{$G(Y,\bx')=0$ is reducible in $\Q[Y]$}\}\\
       \ll_{n,\deg G}( \log (\|G\|+2))^{e(n-1)}B^{(n-1)-1/2},
       \end{multline*}
  for some  $e(n-1)$. 
 
 Suppose $F(Y,X_1,\ldots,X_n)$ satisfies the hypotheses of the proposition. We split up the count according to those $x_n$ such that $F(Y,X_1,...,X_{n-1},x_n)$ is reducible/irreducible, and in the first case we apply Theorem \ref{thm_CDHNV} with $t$ playing the role of $x_n$. (It may be that $\deg_{X_n}F =0$, but recall that Theorem \ref{thm_CDHNV} is also applicable in this case.) This leads to:
\begin{multline*}
\#\{x_{n}\in [-B,B]\cap \Z:\text{$F(Y,X_{1},...,X_{n-1},x_{n})$ is reducible in $\Q[Y,X_1,\ldots,X_{n-1}]$}\}
    \\
    \ll_{n,D}(\log (\|F\|+2))^{h(n)} B^{\frac{1}{2}}.
\end{multline*}
Now if $F(Y,X_{1},...,X_{n-1},x_{n})$ is reducible  in $\Q[Y,X_1,\ldots,X_{n-1}]$, then  $F(Y,x_{1},...,x_{n-1},x_{n})$ is   reducible in $\Q[Y]$ for any choice of $x_{1},...,x_{n-1}$, so we count these trivially. Hence the overall contribution to (\ref{F_solvable_reducible}) of the $x_n \in [-B,B] \cap \Z$ such that $F(Y,X_{1},...,X_{n-1},x_{n})$ is reducible is bounded by 
\[  \ll_{n,D}(\log (\|F\|+2))^{h(n)} B^{(n-1)+\frac{1}{2}}
\ll_{n,D}(\log (\|F\|+2))^{h(n)} B^{n-\frac{1}{2}},\]
which is acceptable.

On the other hand, suppose $x_n \in [-B,B] \cap \Z$ is fixed and $F(Y,X_{1},...,X_{n-1},x_{n})$ is irreducible in $\Q[Y,X_1,\ldots,X_{n-1}]$. To apply the inductive step,  observe that  $\deg_YF(Y,X_{1},...,X_{n-1},x_{n}) \geq 1,$ since $F$ has constant-leading-coefficient in $Y$.
 Then by the inductive hypothesis, 
\[
\#\{\bfx'\in [-B,B]^{n-1} \cap \Z^{n-1}: \text{$F(Y,\bfx',x_{n})$ is reducible in $\Q[Y]$}\}\ll_{n, D}(\log (\|F_{x_n}\|+2))^{e(n-1)} B^{n-1-\frac{1}{2}}.
\]
Here $\|F_{x_n}\|$ denotes the norm of the polynomial $F(Y,X_1,...,X_{n-1},x_n)$ and satisfies the relation  
\[ \log(\|F_{x_n}\|+2) \ll_{n,D} \log(\|F\|+2)+\log B.\]
The number of $x_n$ which lead to this case is trivially at most $\ll B$, 
so the overall contribution from all $x_{n}$ for which $F(Y,X_{1},...,X_{n-1},x_{n})$ is irreducible in $\Q[Y,X_1,\ldots,X_{n-1}]$ is $\ll_{n, D}(\log (\|F\|+2)+\log B)^{e(n-1)} B^{n-\frac{1}{2}}$.

In total,
this argument has proved that for all $B \geq 1$, the conclusion of the proposition holds,  
with $e(n) = \max\{h(n),e(n-1)\}.$ This completes the proof of the proposition, and hence of Theorem \ref{thm_polylog}.

\section{A question about uniformity}\label{sec_uniformity_question}

Recall from \S \ref{sec_lit_typeI} that a distinctive achievement in the study of thin sets of type I is that many results are uniform; see Theorems \ref{thm_affine_hypersurface} and \ref{thm_UDGC_degneq3}.
In light of this, it is tempting to ask whether a similar phenomenon holds for thin sets of type II. That is, should a uniform version of Conjecture \ref{conj_Serre_projII} be true for projective thin sets of type II? Or, in the affine setting, should a uniform bound $N^\cov_{\Abb^n}(F,B) \ll_{n,D} B^{n-1}s(B)$ hold for all $F$ in some suitable class of polynomials, and an appropriate choice of ``small factor'' $s(B)?$   In this section, we demonstrate that the latter question appears to be subtle, and to have different features from known results in the type I setting.

\subsection{A restricted counting function}
First, it is illuminating to introduce temporarily a restricted counting function that avoids an essential feature of the type II setting.
Let $F(Y,\bX)$ be a polynomial of the form (\ref{BonPie_shape}) or (\ref{BCSSV_class}), so that it is weighted homogeneous with weights $(e,1,\ldots,1)$, or has top degree polynomial $F_\mathrm{top}$ that is weighted homogeneous with weights $(e,1,\ldots,1)$. 
Recall the counting function $N^\cov_{\Abb^n}(F,B)$ for the associated affine thin set of type II, as in (\ref{Ncov_aff_dfn}). 
Define now the restricted count 
\[ [N^\cov_{\Abb^n}]_{\mathrm{res}} (F;B^e,B)  :=\{y\in [-B^e,B^e]\cap \Z, \bfx \in [-B,B]^n\cap \Z^n:   F(y,\bfx)=0 \}.\]
 Thus this restricts to solutions with $|y|\leq B^e$, in contrast to the count $N^\cov_{\Abb^n}(F,B)$, which \emph{a priori} allows $|y| \ll \|F\| B^e$.  In this sense, the restricted count attempts to mimic the type I setting, in which all variables are governed only by the parameter $B$.
 
 The methods of \cite{BonPie24} can be applied to prove  
\beq\label{res_strength} 
[N^\cov_{\Abb^n}]_{\mathrm{res}} (F;B^e,B)\ll_{n,m,d,e}B^{n-1+1/(n+1)}(\log B)^{n/(n+1)},
\eeq
with implicit constant independent of $\|F\|$. But for the original count $N^\cov_{\Abb^n}(F,B)$, polylogarithmic dependence on $\|F\|$ cannot be ruled out by those methods; this is explained in the correction \cite{BonPie24cor}. 

The restricted count can nevertheless be a useful tool for studying the original count, but at the cost of introducing 
polynomial dependence on $\|F\|$. Indeed, for any polynomial 
 $F(Y,\bX)$ of the class (\ref{BCSSV_class}), \cite[Thm. 1.3]{BCSSV25x} proved that for some absolute constant $\ga$,
 \[ 
[N^\cov_{\Abb^n}]_{\mathrm{res}} (F;B^e,B) \ll_{n,d}  \begin{cases} B^{n-1} & \text{if $d \geq 5$}\\
B^{n-1} (\log B)^{\ga} & \text{if $d =4$}\\
B^{n-1+(2/\sqrt{d}-1)}(\log B)^{\ga} & \text{if $d =2,3$,}
\end{cases}\] 
  with implicit constant independent of $\|F\|$. There exists   a combinatorial constant $c_{n,d,e}$ such that  
\beq\label{restricted_pass_to_original} 
N^\cov_{\Abb^n}(F,B) \ll [N^\cov_{\Abb^n}]_{\mathrm{res}}(F;c_{n,d,e}\|F\|B^e,B).
\eeq
Therefore, the uniform bound for the restricted count immediately implies   the bound (\ref{BCSSV_affine}) for the original count, with polynomial dependence on $\|F\|$. To achieve at most polylogarithmic dependence on $\|F\|$ in a bound for $N^\cov_{\Abb^n}(F,B)$, different methods must be employed, such as in \cite{BPW25x}.

  \subsection{Examples that fail certain uniform upper bounds}
We now demonstrate via counterexamples that a particular shape of conjectural uniform bound for $N^\cov_{\Abb^n}(F,B)$ for affine thin sets of type II can be false; this requires specifying the small factor $s(B)$ as well. (For clarity, we note that \cite{Ser97} did not address uniformity questions, and there is no contradiction to statements there.)

\begin{example}[Sums of two squares]
    Consider the affine thin set $M_k$ defined for a given integer $k$ by $$M_k = \{x_1 \in \Z: \exists y \in \Z: y^2 + x_1^2 - k = 0\} \subset \A^1(\Z).$$
    For $B\geq k^{1/2}$, 
    $$N_{\Abb^1}(M_k,B) = \tfrac{1}{2} r(k),$$
    where $r(k)$ is the number of representations of $k$ as the sum of two squares. If $k$ is squarefree, we can compute $r(k)$ explicitly as 
    \beq\label{rk_reps_sums_of_squares} 
    r(k) = 4 \cdot 2^{\om(k)},
    \eeq
    if every prime divisor $p|k$ has $p \con 1 \modd{4}$, and $r(k)=0$ otherwise; see 
    for example \cite[Theorem 2.5]{MorWag06}.
    In particular, if we define $$k = \prod_{\substack{p\leq \log B \\ p\equiv 1 \bmod 4}} p ,$$
    then note that $k \sim B^{1/2}$ as $B \maps \infty$, and $\om(k) \sim \tfrac{1}{2} \log B/\log \log B$, by the Siegel-Walfisz theorem. 
   Consequently we can calculate that as $B \maps \infty,$
    \begin{equation}\label{eq: counter ex uniform Serre}N_{\Abb^1}(M_k,B) = 2^{\frac{\log B}{\log\log B}(1/2+o(1))}.\end{equation}
 A uniform version of Conjecture \ref{conj_Serre_projII} for thin sets of type II in $\Abb^1$ with polylogarithmic small factor, say $s(B) \ll (\log B)^c$, would ask whether for some constant $c$,
 $$N_{\Abb^1}(M_k,B) = O((\log B)^c)$$
    holds independent of the choice of $k$. This fails as   $B \maps \infty$, by (\ref{eq: counter ex uniform Serre}). However, we observe that a putative uniform bound of the form $$N_{\Abb^1}(M_k,B) = O_{\ep}(B^{\ep})$$
    is not violated by this example. 
\end{example}
\begin{example}[A multi-dimensional generalization]
    Define for a given integer $k$ the affine thin set 
    $$M_k = \{(x_1,\ldots,x_n)\in \Z^n: y^2 + x_1^2 - k(x_2+\ldots+x_n)=0\}.$$
    (Note that in the terminology of \cite{BPW25x}, while $F_k(y,x_1,\ldots,x_n):=y^2 +x_1^2 - k(x_2+\ldots+x_n)$ is a polynomial in $n+1$ variables, it is not a $(1,n)$-allowable polynomial.) 
    For $B\gg k$, we see that 
    $$N_{\Abb^n}(M_k,B) \gg \sum_{(x_2,\ldots,x_n) \in [-B,B]^{n-1}} r(k(x_2+\ldots+x_n))\gg B^{n-2} \sum_{|z|\ll B} r(kz).$$
    Here $r(k)$ again denotes the number of representations of $k$ as the sum of two squares.
    Now, $r(k)$ is not a totally multiplicative function, but it does satisfy that 
    $$r(m_1m_2) = \sum_{d\mid \gcd(m_1,m_2)}\mu(d)\chi(d) r(m_1/d)r(m_2/d),$$
    where $\chi$ is the nontrivial quadratic character mod 4 (see \cite[Section 6]{HB03}; this is also an instance of a Busche-Ramanujan identity). Thus, \begin{equation}\label{eq: uniformity counter-ex multi dimensional}\sum_{|z|\ll B} r(kz) = \sum_{d\mid k} \mu(d)\chi(d)r(k/d)\sum_{\substack{|z|\ll B\\ d\mid z}} r(z/d).\end{equation}
  By counting the number of lattice points in a circle of radius $X^{1/2}$, 
    $$\sum_{x\leq X} r(x) = \pi X + O(X^{1/2}).$$
   Plugging the above into (\ref{eq: uniformity counter-ex multi dimensional}), we have that 
    \beq\label{Gauss_circle_remainder}
    \sum_{|z|\ll B} r(kz) = \pi B \sum_{d\mid k} \frac{\mu(d)\chi(d)r(k/d)}{d}  + O(B^{1/2}\sum_{d\mid k} r(k/d)d^{-1/2}).
    \eeq
    Now define $k$ to be 
    $$k = \prod_{\substack{p\equiv 1 \bmod 4\\ \log B/2 \leq p\leq \log B}}p.$$
    Then $k \sim  B^{1/4}$ and $\om(k) \sim \tfrac{1}{4}\log B/ \log \log B$, by the Siegel-Walfisz theorem. Moreover, (\ref{rk_reps_sums_of_squares}) shows that $r(k/d) \leq r(k)$ for each $d|k$. We also observe for future reference that $\sum_{p|k}1/p \ll (\log B)^{-1} \om(k) \ll (\log \log B)^{-1}$.
    
  We may rewrite (\ref{Gauss_circle_remainder})  as
    \[  \sum_{|z|\ll B} r(kz)  = \pi B r(k) +E_1 + E_2\]
    in which 
    \begin{align*} 
   | E_1 |& \ll  B \sum_{\substack{d\mid k\\ d>1}} \frac{r(k/d)}{d} ,
    \\
    |E_2| & \ll r(k)B^{1/2}\sum_{d\mid k}  d^{-1/2}  \ll r(k) d(k) B^{1/2},
    \end{align*}
     denoting the divisor function by $d(k)$.
  By our previous observation,
    \begin{align*}|E_1| & \ll Br(k)\left(\sum_{p\mid k} \frac{1}{p} + \sum_{p_1p_2\mid k} \frac{1}{p_1p_2}+\hdots \right) \\ 
    & \ll Br(k) \left(\frac{1}{\log\log  B}+ \frac{1}{(\log\log B)^2}+\hdots \right)
    \ll Br(k) \left(\frac{1}{\log\log  B-1} \right).
    \end{align*}
    Additionally, since $k$ is squarefree so that $d(k) \ll 2^{\om(k)} \ll B^{1/ \log \log B} \ll B^{1/4}$, say, we can write $|E_2| \ll B r(k) \cdot B^{-1/4}$.
    Consequently,
      \[  \sum_{|z|\ll B} r(kz)  = \pi B r(k)(1+ O(( \log \log B)^{-1}))\] 
      as $B \maps \infty$. 
     This shows that
   $$N_{\Abb^n}(M_k,B) \gg B^{n-1} r(k)(\pi + o(1)) \gg 2^{\frac{\log B}{\log\log B}(1/4+o(1))} B^{n-1}.$$
    This example demonstrates that for any $c$, a bound of the form $$N_{\Abb^n}(M_k,B) = O_{n,\deg(M_k)}(B^{n-1}(\log B)^{c})$$
    that is independent of $k$, fails as $B \maps \infty$. It does not violate a putative bound of the form $$N_{\Abb^n}(M_k,B) = O_{n,\deg(M_k),\ep}(B^{n-1+\ep}).$$

\end{example}
These counterexamples have demonstrated Theorem \ref{thm_counterex}.

 \subsection*{Acknowledgements}
 
The authors thank Tim Browning and Raf Cluckers for helpful references, and detailed comments on an earlier draft. L.P.  has been partially supported during portions of this research by NSF DMS-2200470, a Joan and Joseph Birman Fellowship, a Simons Fellowship, and a Guggenheim Fellowship, and thanks the Hausdorff Center for Mathematics for hosting several research periods as a Bonn Research Chair; the Mittag-Leffler Institute for hosting a research period in 2024; Rhodes House in 2025. 
K.W. visited Duke in 2023 with funding from NSF RTG-2231514, which supports the Number Theory group at Duke University. K.W. is partially supported by NSF under DGE-2039656 and DMS-2502864.

\bibliographystyle{alpha}

\bibliography{NoThBib}

\newcommand{\etalchar}[1]{$^{#1}$}
\begin{thebibliography}{CDH{\etalchar{+}}23}

\bibitem[BCK25]{BCK25}
G.~Binyamini, R.~Cluckers, and F.~Kato.
\newblock Sharp bounds for the number of rational points on algebraic curves
  and dimension growth, over all global fields.
\newblock {\em Proc. Lond. Math. Soc. (3)}, 130(1):Paper No. e70016, 20, 2025.

\bibitem[BCLP23]{BCLP23}
A.~Bucur, A.~C. Cojocaru, M.~N. Lal\'{i}n, and L.~B. Pierce.
\newblock Geometric generalizations of the square sieve, with an application to
  cyclic covers.
\newblock {\em Mathematika}, 69:106--154, 2023.

\bibitem[BCN24]{BCN24}
G.~Binyamini, R.~Cluckers, and D.~Novikov.
\newblock Bounds for rational points on algebraic curves, optimal in the
  degree, and dimension growth.
\newblock {\em Int. Math. Res. Not. IMRN}, (11):9256--9265, 2024.

\bibitem[BCS{\etalchar{+}}25]{BCSSV25x}
T.~Buggenhout, R.~Cluckers, P.~Salberger, T.~Santens, and F.~Vermeulen.
\newblock Serre's question on thin sets in projective space.
\newblock (preprint arXiv:2506.13471) 2025.

\bibitem[BHB05]{BroHB05}
T.~D. Browning and D.~R. Heath-Brown.
\newblock Counting rational points on hypersurfaces.
\newblock {\em Journal f\"{u}r die reine und angewandte Mathematik},
  2005(584):83--115, 2005.

\bibitem[BHB06a]{BroHB06a}
T.~D. Browning and D.~R. Heath-Brown.
\newblock The density of rational points on non-singular hypersurfaces, {I}.
\newblock {\em Bull. London Math. Soc.}, 38(3):401--410, 2006.

\bibitem[BHB06b]{BroHB06b}
T.~D. Browning and D.~R. Heath-Brown.
\newblock The density of rational points on non-singular hypersurfaces, {II}.
\newblock {\em Proc. London Math. Soc. (3)}, 93(2):273--303, 2006.
\newblock With an appendix by J. M. Starr.

\bibitem[BHBS06]{BHBS06}
T.~D. Browning, D.~R. Heath-Brown, and P.~Salberger.
\newblock Counting rational points on algebraic varieties.
\newblock {\em Duke Math. Journal}, 132:545--578, 2006.

\bibitem[BKW25]{BKW25x}
D.~Bonolis, E.~Kowalski, and K.~Woo.
\newblock Stratification for exponential sums, revisited. {With an appendix by
  A. Forey, J. Fres\'{a}n, and E. Kowalski}.
\newblock (preprint, arXiv:2506.18299) 2025.

\bibitem[Bom78]{Bom78}
E.~Bombieri.
\newblock On exponential sums in finite fields. {II}.
\newblock {\em Invent. Math.}, 47(1):29--39, 1978.

\bibitem[Bon21]{Bon21}
D.~Bonolis.
\newblock A polynomial sieve and sums of {D}eligne type.
\newblock {\em Int. Math. Res. Not. IMRN}, (2):1096--1137, 2021.

\bibitem[BP89]{BomPil89}
E.~Bombieri and J.~Pila.
\newblock The number of integral points on arcs and ovals.
\newblock {\em Duke Math. J.}, 59(2):337--357, 1989.

\bibitem[BP24]{BonPie24}
D.~Bonolis and L.~B. Pierce.
\newblock Application of a polynomial sieve: beyond separation of variables.
\newblock {\em Algebra \& Number Theory}, 18(8):1515--1556, 2024.

\bibitem[BP26]{BonPie24cor}
D.~Bonolis and L.~B. Pierce.
\newblock Correction to ``{A}pplication of a polynomial sieve: beyond
  separation of variables''.
\newblock {\em Algebra \& Number Theory (accepted)}, arXiv:2209.02494
  (appended), 2026.

\bibitem[BPW25]{BPW25x}
D.~Bonolis, L.~B. Pierce, and K.~Woo.
\newblock Counting integral points in thin sets of type {II}: singularities,
  sieves, and stratification.
\newblock {\em (preprint, arXiv:2505.11226)}, 2025.

\bibitem[BPW26]{BPW25x_gen}
D.~Bonolis, L.~B. Pierce, and K.~Woo.
\newblock Genuine and strongly genuine polynomials: with an application to the
  persistence of {G}alois groups under specialization.
\newblock (preprint) 2026.

\bibitem[Bro03a]{Bro03N}
N.~Broberg.
\newblock Rational points on finite covers of {$\mathbb{P}^1$} and
  {$\mathbb{P}^2$}.
\newblock {\em J. Number Theory}, 101(1):195--207, 2003.

\bibitem[Bro03b]{Bro03a}
T.~D. Browning.
\newblock A note on the distribution of rational points on threefolds.
\newblock {\em Q. J. Math.}, 54(1):33--39, 2003.

\bibitem[CCDN20]{CCDN20}
W.~Castryck, R.~Cluckers, P.~Dittmann, and K.~H. Nguyen.
\newblock The dimension growth conjecture, polynomial in the degree and without
  logarithmic factors.
\newblock {\em Algebra \& Number Theory}, 14(8):2261--2294, 2020.

\bibitem[CDH{\etalchar{+}}23]{CDHNV23x}
R.~Cluckers, P.~D\`{e}bes, Y.~Hendel, K.~H. Nguyen, and F.~Vermeulen.
\newblock Improvements on {Dimension Growth results and effective Hilbert's
  Irreducibility Theorem}.
\newblock {\em arXiv:2311.16871}, 2023.

\bibitem[CG25]{CluGla25}
R.~Cluckers and I.~Glazer.
\newblock Eventual tightness of projective dimension growth bounds: quadratic
  in the degree.
\newblock {\em J. Number Theory}, 276:72--80, 2025.

\bibitem[Coh81]{Coh81}
S.~D. Cohen.
\newblock The distribution of {G}alois groups and {H}ilbert's irreducibility
  theorem.
\newblock {\em Proc. London Math. Soc. (3)}, 43(2):227--250, 1981.

\bibitem[Deh26]{Deh26x}
L.~Dehennin.
\newblock Improved, sublinear projective {Schwartz-Zippel} and subquadratic
  dimension growth bounds.
\newblock arXiv:2602.23942, 2026.

\bibitem[EH16]{EisHar16}
D.~Eisenbud and J.~Harris.
\newblock {\em 3264 and all that---a second course in algebraic geometry}.
\newblock Cambridge University Press, Cambridge, 2016.

\bibitem[GKZ08]{GKZ08}
I.~M. Gelfand, M.~M. Kapranov, and A.~V. Zelevinsky.
\newblock {\em Discriminants, Resultants and Multidimensional Determinants}.
\newblock Modern Birkh\"{a}user Classics. Birkh\"{a}user Boston, Inc., Boston,
  MA, 2008.
\newblock Reprint of the 1994 edition.

\bibitem[Gro66]{EGAIV_part3}
A.~Grothendieck.
\newblock {EGA IV, \'El\'ements de g\'eom\'etrie alg\'ebrique : {IV.} {\'Etude}
  locale des sch\'emas et des morphismes de sch\'emas, {Troisi\`eme} partie}.
\newblock {\em Publications Math\'ematiques de l'IH\'ES}, 28:5--255, 1966.

\bibitem[Har77]{Har77}
R.~Hartshorne.
\newblock {\em Algebraic Geometry}.
\newblock Springer-Verlag, New York-Heidelberg, 1977.
\newblock Graduate Texts in Mathematics, No. 52.

\bibitem[HB83]{HB83}
D.~R. Heath-Brown.
\newblock Cubic forms in ten variables.
\newblock {\em Proc. London Math. Soc.}, 47:225--257, 1983.

\bibitem[HB94]{HB94}
D.~R. Heath-Brown.
\newblock The density of rational points on non-singular hypersurfaces.
\newblock {\em Proc. Indian Acad. Sci. (Math. Sci.)}, 104:13--29, 1994.

\bibitem[HB02]{HB02}
D.~R. Heath-Brown.
\newblock The density of rational points on curves and surfaces.
\newblock {\em Ann. of Math.}, 155:553--595, 2002.

\bibitem[HB03]{HB03}
D.~R. Heath-Brown.
\newblock Linear relations amongst sums of two squares.
\newblock In {\em Number theory and algebraic geometry}, volume 303 of {\em
  London Math. Soc. Lecture Note Ser.}, pages 133--176. Cambridge Univ. Press,
  Cambridge, 2003.

\bibitem[HBP12]{HBPie12}
D.~R. Heath-Brown and L.~B. Pierce.
\newblock Counting rational points on smooth cyclic covers.
\newblock {\em J. Number Theory}, 132(8):1741--1757, 2012.

\bibitem[Lan83]{Lan83}
S.~Lang.
\newblock {\em Fundamentals of {D}iophantine {G}eometry}.
\newblock Springer-Verlag, New York, 1983.

\bibitem[LO77]{LagOdl77}
J.~C. Lagarias and A.~M. Odlyzko.
\newblock Effective versions of the {C}hebotarev density theorem.
\newblock In {\em Algebraic number fields: {$L$}-functions and {G}alois
  properties ({P}roc. {S}ympos., {U}niv. {D}urham, {D}urham, 1975)}, pages
  409--464. Academic Press, London, 1977.

\bibitem[Mun09]{Mun09}
R.~Munshi.
\newblock Density of rational points on cyclic covers of {$\mathbb{P}^n$}.
\newblock {\em Journal de Th\'{e}orie des Nombres de Bordeaux}, 21:335--341,
  2009.

\bibitem[MW06]{MorWag06}
C.~Moreno and S.~Wagstaff.
\newblock {\em Sums of Squares of Integers}.
\newblock Discrete Mathematics and Its Applications. Chapman and Hall/CRC, Boca
  Raton, London, New York, 2006.

\bibitem[Pil95]{Pil95}
J.~Pila.
\newblock Density of integral and rational points on varieties.
\newblock In {\em Columbia University Number Theory Seminar (New York, 1992)},
  volume 228 of {\em Ast\'{e}risque}, pages 183--187. Soc. Math. France,
  Montrouge, 1995.

\bibitem[Rei88]{Rei88}
M.~Reid.
\newblock {\em Undergraduate Algebraic Geometry}, volume~12 of {\em London
  Mathematical Society Student Texts}.
\newblock Cambridge University Press, Cambridge, 1988.

\bibitem[Sal15]{Sal15}
P.~Salberger.
\newblock Uniform bounds for rational points on cubic hypersurfaces.
\newblock In {\em Arithmetic and Geometry}, London Mathematical Society Lecture
  Note Series, page 401–421. Cambridge University Press, 2015.

\bibitem[Sal23]{Sal23}
P.~Salberger.
\newblock Counting rational points on projective varieties.
\newblock {\em Proc. Lond. Math. Soc. (3)}, 126(4):1092--1133, 2023.

\bibitem[Ser81]{Ser81}
J.-P. Serre.
\newblock Quelques applications du th\'{e}or\`eme de densit\'{e} de
  {C}hebotarev.
\newblock {\em Inst. Hautes \'{E}tudes Sci. Publ. Math.}, (54):323--401, 1981.

\bibitem[Ser92]{Ser92}
J.-P. Serre.
\newblock {\em Topics in {G}alois theory}, volume~1 of {\em Research Notes in
  Mathematics}.
\newblock Jones and Bartlett Publishers, Boston, MA, 1992.
\newblock Lecture notes prepared by H. Darmon.

\bibitem[Ser97]{Ser97}
J.-P. Serre.
\newblock {\em Lectures on the {M}ordell-{W}eil theorem}.
\newblock Aspects of Mathematics. Friedr. Vieweg \& Sohn, Braunschweig, 3rd
  edition, 1997.
\newblock Translated from the French and edited by M. Brown from notes by M.
  Waldschmidt.

\bibitem[Sha13]{Sha13}
I.~R. Shafarevich.
\newblock {\em Basic Algebraic Geometry I}.
\newblock Springer, Heidelberg, 3rd edition, 2013.
\newblock Varieties in projective space.

\bibitem[{Sta}26]{StaPro}
The {Stacks project authors}.
\newblock The stacks project.
\newblock \url{https://stacks.math.columbia.edu}, 2026.

\bibitem[Ver24]{Ver24}
F.~Vermeulen.
\newblock Dimension growth for affine varieties.
\newblock {\em Int. Math. Res. Not. IMRN}, (15):11464--11483, 2024.

\bibitem[Wal15]{Wal15}
M.~N. Walsh.
\newblock Bounded rational points on curves.
\newblock {\em Int. Math. Res. Not. IMRN}, (14):5644--5658, 2015.

\end{thebibliography}

\end{document}